\documentclass[a4paper,11pt,reqno]{amsart}

\usepackage{dsfont,amssymb}

\usepackage{hyperref}
\usepackage[nobysame,alphabetic,initials,msc-links]{amsrefs}

\DeclareFontFamily{U}  {MnSymbolC}{}
\DeclareFontShape{U}{MnSymbolC}{m}{n}{
    <-6>  MnSymbolC5
   <6-7>  MnSymbolC6
   <7-8>  MnSymbolC7
   <8-9>  MnSymbolC8
   <9-10> MnSymbolC9
  <10-12> MnSymbolC10
  <12->   MnSymbolC12}{}
\DeclareSymbolFont{MnSyC}         {U}  {MnSymbolC}{m}{n}
\DeclareMathSymbol{\lefthalfcup}{\mathbin}{MnSyC}{183}
\DeclareMathSymbol{\righthalfcup}{\mathbin}{MnSyC}{184}

\usepackage[all]{xy}

\DefineSimpleKey{bib}{how}
\renewcommand{\PrintDOI}[1]{%
  \href{http://dx.doi.org/#1}{{\tt DOI:#1}}%
}
\renewcommand{\eprint}[1]{#1}
\BibSpec{book}{%
    +{}  {\PrintPrimary}                {transition}
    +{.} { \PrintDate}                  {date}
    +{.} { \textit}                     {title}
    +{.} { }                            {part}
    +{:} { \textit}                     {subtitle}
    +{,} { \PrintEdition}               {edition}
    +{}  { \PrintEditorsB}              {editor}
    +{,} { \PrintTranslatorsC}          {translator}
    +{,} { \PrintContributions}         {contribution}
    +{,} { }                            {series}
    +{,} { \voltext}                    {volume}
    +{,} { }                            {publisher}
    +{,} { }                            {organization}
    +{,} { }                            {address}
    +{,} { }                            {status}
    +{,} { \PrintDOI}                   {doi}
    +{,} { \PrintISBNs}                 {isbn}
    +{}  { \parenthesize}               {language}
    +{}  { \PrintTranslation}           {translation}
    +{;} { \PrintReprint}               {reprint}
    +{.} { }                            {note}
    +{.} {}                             {transition}
    +{}  {\SentenceSpace \PrintReviews} {review}
}
\BibSpec{article}{%
    +{}  {\PrintAuthors}                {author}
    +{,} { \textit}                     {title}
    +{.} { }                            {part}
    +{:} { \textit}                     {subtitle}
    +{,} { \PrintContributions}         {contribution}
    +{.} { \PrintPartials}              {partial}
    +{,} { }                            {journal}
    +{}  { \textbf}                     {volume}
    +{}  { \PrintDatePV}                {date}
    +{,} { \issuetext}                  {number}
    +{,} { \eprintpages}                {pages}
    +{,} { }                            {status}
    +{,} { \PrintDOI}                   {doi}
    +{,} { \eprint}        {eprint}
    +{}  { \parenthesize}               {language}
    +{}  { \PrintTranslation}           {translation}
    +{;} { \PrintReprint}               {reprint}
    +{.} { }                            {note}
    +{.} {}                             {transition}
    +{}  {\SentenceSpace \PrintReviews} {review}
}
\BibSpec{collection.article}{%
    +{}  {\PrintAuthors}                {author}
    +{,} { \textit}                     {title}
    +{.} { }                            {part}
    +{:} { \textit}                     {subtitle}
    +{,} { \PrintContributions}         {contribution}
    +{,} { \PrintConference}            {conference}
    +{}  {\PrintBook}                   {book}
    +{,} { }                            {booktitle}
    +{,} { \PrintDateB}                 {date}
    +{,} { pp.~}                        {pages}
    +{,} { }                            {publisher}
    +{,} { }                            {organization}
    +{,} { }                            {address}
    +{,} { }                            {status}
    +{,} { \PrintDOI}                   {doi}
    +{,} { \eprint}        {eprint}
    +{}  { \parenthesize}               {language}
    +{}  { \PrintTranslation}           {translation}
    +{;} { \PrintReprint}               {reprint}
    +{.} { }                            {note}
    +{.} {}                             {transition}
    +{}  {\SentenceSpace \PrintReviews} {review}
}
\BibSpec{misc}{%
  +{}{\PrintAuthors}  {author}
  +{,}{ \textit}      {title}
  +{.}{ }             {how}
  +{}{ \parenthesize} {date}
  +{,} { available at \eprint}        {eprint}
  +{,}{ available at \url}{url}
  +{,}{ }             {note}
  +{.}{}              {transition}
}

\numberwithin{equation}{section}

\newtheorem{theorem}{Theorem}[section]
\newtheorem{corollary}[theorem]{Corollary}
\newtheorem{lemma}[theorem]{Lemma}
\newtheorem{proposition}[theorem]{Proposition}
\theoremstyle{remark}
\newtheorem{remark}[theorem]{Remark}

\theoremstyle{definition}

\newcommand{\bp}{\begin{proof}}
\newcommand{\ep}{\end{proof}}

\mathchardef\mhyph="2D

\DeclareMathOperator{\Ad}{Ad}
\DeclareMathOperator{\Aut}{Aut}
\DeclareMathOperator{\End}{End}
\DeclareMathOperator{\Hom}{Hom}

\DeclareMathOperator{\Irr}{Irr}

\DeclareMathOperator{\Rep}{Rep}

\DeclareMathOperator{\rk}{rk}

\DeclareMathOperator{\Tr}{Tr}
\DeclareMathOperator{\Prim}{Prim}

\DeclareMathOperator{\Obj}{Obj}

\newcommand{\C}{\mathbb{C}}

\newcommand{\R}{\mathbb{R}}
\newcommand{\T}{\mathbb{T}}
\newcommand{\Z}{\mathbb{Z}}
\newcommand{\A}{\mathcal{A}}

\newcommand{\CC}{\mathcal{C}}

\newcommand{\E}{\mathcal{E}}

\newcommand{\PP}{\mathcal{P}}
\newcommand{\DD}{\mathcal{D}}
\newcommand{\U}{\mathcal{U}}

\newcommand{\id}{\mathrm{id}}
\newcommand{\Id}{\mathrm{Id}}

\newcommand{\Hilb}{\mathrm{Hilb}}

\newcommand{\su}{\mathfrak{su}}
\newcommand{\so}{\mathfrak{so}}

\newcommand{\SU}{\mathrm{SU}}
\newcommand{\SL}{\mathrm{SL}}
\newcommand{\Spin}{\mathrm{Spin}}
\newcommand{\Sp}{\mathrm{Sp}}
\newcommand{\SO}{\mathrm{SO}}
\newcommand{\rU}{\mathrm{U}}
\newcommand{\un}{{\mathds 1}}

\newcommand{\absv}[1]{\left| #1 \right|}

\newcommand{\Cliff}{\mathrm{Cl}}
\newcommand{\half}{\frac{1}{2}}
\newcommand{\odd}{\mathrm{odd}}
\newcommand{\even}{\mathrm{even}}
\newcommand{\Dhat}{\hat{\Delta}}

\usepackage{wasysym}
\newcommand{\circt}%
{\mathbin{%
\mathchoice
{\ooalign{$\ocircle$\cr\hidewidth\raise-.15ex\hbox{$\scriptstyle\top\mkern2.05mu$}\cr}}
{\ooalign{$\ocircle$\cr\hidewidth\raise-.15ex\hbox{$\scriptstyle\top\mkern2.05mu$}\cr}}
{\ooalign{$\scriptstyle\ocircle$\cr\hidewidth\raise-.12ex\hbox{$\scriptscriptstyle\top\mkern1mu$}\cr}}
{\ooalign{$\scriptstyle\ocircle$\cr\hidewidth\raise-.12ex\hbox{$\scriptscriptstyle\top\mkern1mu$}\cr}}
}}

\begin{document}

\title[non-Kac $\SU(n)$-type quantum groups]{Classification of non-Kac compact quantum groups of $\SU(n)$ type}

\author[S. Neshveyev]{Sergey Neshveyev}

\email{sergeyn@math.uio.no}

\address{Department of Mathematics, University of Oslo,
P.O. Box 1053 Blindern, NO-0316 Oslo, Norway}

\thanks{The research leading to these results has received funding from the European Research Council
under the European Union's Seventh Framework Programme (FP/2007-2013) / ERC Grant Agreement no. 307663
}

\author[M. Yamashita]{Makoto Yamashita}

\email{yamashita.makoto@ocha.ac.jp}

\address{Department of Mathematics, Ochanomizu University,
Otsuka 2-1-1, Bunkyo, 112-8610 Tokyo, Japan}

\thanks{Supported by the Danish National Research Foundation through the Centre
for Symmetry and Deformation (DNRF92), and by JSPS KAKENHI Grant Number 25800058}

\date{May 26, 2014; minor changes August 4, 2015}

\begin{abstract}
We classify up to isomorphism all non-Kac compact quantum groups with the same fusion rules and dimension function as $\SU(n)$. For this we first prove, using categorical Poisson boundary, the following general result. Let $G$ be a coamenable compact quantum group and $K$ be its maximal quantum subgroup of Kac type. Then any dimension-preserving unitary fiber functor $\Rep G\to \Hilb_f$ factors, uniquely up to isomorphism, through $\Rep K$. Equivalently, we have a canonical bijection $H^2(\hat G;\T)\cong H^2(\hat K;\T)$. Next, we classify autoequivalences of the representation categories of twisted $q$-deformations of compact simple Lie groups.
\end{abstract}

\maketitle

\bigskip

\section{Introduction}

In his fundamental paper~\cite{MR943923} on quantization of $\SU(n)$ and the Tannaka--Krein duality for compact quantum groups,  Woronowicz formulated the following  problem: classify quantum groups having the same representation theory, meaning the same fusion rules and dimensions of irreducible representations, as $\SU(n)$. Since then, this and similar questions for other representation rings have been studied by a number of authors, see e.g.,~\citelist{\cite{MR1266253} \cite{MR1629723} \cite{MR1378260} \cite{MR1673475} \cite{MR1809304} \cite{MR2023750} \cite{MR2106933} \cite{MR3275027} \cite{MR3240820}} and references therein.

Despite some success, a complete answer has only been obtained in a few low rank cases. At the current stage this is hardly surprising, and a complete explicit answer to Woronowicz's question is not to be expected. Indeed, such an answer would include understanding of all unitary fiber functors on $\Rep\SU(n)$, which is equivalent to classifying full multiplicity ergodic actions of $\SU(n)$ on C$^*$-algebras~\citelist{\cite{MR990110}\cite{MR1190512}}. As it was known already to Wassermann~\cite{MR990110} (see also the discussion in~\cite{MR2844801}*{pp.~1240--1241}) this, in turn, includes classifying the finite central type factor groups inside~$\SU(n)$ up to conjugacy. For small $n$ there are not many such subgroups, and for $n=2,3$ it is, indeed, possible to classify all full multiplicity ergodic actions~\citelist{\cite{MR948104}\cite{MR996457}}. However, the problem rapidly becomes unfeasible as~$n$ grows larger. Since fiber functors on $\Rep\SU(n)$ lead to Kac quantum groups, we therefore should not expect an explicit answer to Woronowicz's question for the class of compact quantum groups of Kac type. Note that nevertheless for genuine compact groups we have a complete answer: by a result of McMullen~\cite{MR733774}, a compact group with the same fusion rules as for $\SU(n)$ is itself isomorphic to $\SU(n)$, see also~\citelist{\cite{MR1209960}\cite{MR3207584}}.

We can also forget about compactness and try  to describe all cosemisimple Hopf algebras with corepresentation theory of $\SU(n)$. For $n=2$ this was done by Podle{\'s} and M{\"u}ller~\cite {MR1629723}, and by Bichon~\cite{MR2023750} without the restriction on the dimension function, extending a result of Banica in the compact case~\cite{MR1378260}. For $n=3$, a complete classification of such Hopf algebras was obtained by Ohn~\citelist{\cite{MR1673475} \cite{MR2106933}}, where he obtained a long list of various multiparametric deformations through a formidable amount of direct computations. Doing anything similar for $n\ge4$ seems like an overwhelming task, even with computer assistance.

Our main result is that if we stay away from the Kac and noncompact cases, the question of Woronowicz has a very simple answer: the only quantum groups we have are~$\SU_q(n)$ for $0<q<1$, the categorical twists $\SU^\tau_q(n)$ of $\SU_q(n)$ studied in our previous paper~\cite{MR3340190}, and the deformations of such quantum groups by $2$-cocycles on the dual of the maximal torus, see Theorems~\ref{thm:answer-Woronowicz-non-Kac} and~\ref{thm:answer-Woronowicz-non-Kac2} for the precise statement. To put it differently, all such compact quantum groups are obtained by quantization of one of the Poisson--Lie group structures on $\SU(n)$~\cite{MR1116413} and by twisting by a $3$-cocycle on the Pontryagin dual of the center.

It is worth remembering that, on the purely algebraic level, compact quantum groups are simply Hopf $*$-algebras generated by matrix coefficients of their finite dimensional unitary corepresentations. Therefore our result classifies all such Hopf $*$-algebras with corepresentation theory of~$\SU(n)$ and noninvolutive antipode. For $n=3$ this implies that a majority of Hopf algebras in the list of Ohn are of Kac type, do not admit a compatible $*$-structure, or have nonunitarizable corepresentations, which can also be checked by a careful inspection of his classification.

\smallskip

In view of the Tannaka--Krein duality, the classification problem can be divided into three parts:
\begin{enumerate}
\item[--] classification of rigid C$^*$-tensor categories $\CC$ with fusion rules of $\SU(n)$;
\item[--] classification of monoidal autoequivalences of $\CC$;
\item[--] classification of unitary fiber functors $\CC\to\Hilb_f$ inducing the classical dimension function on the representation ring of $\SU(n)$.
\end{enumerate}

The first problem was solved in the purely algebraic setting by Kazhdan and Wenzl~\cite{MR1237835}, and the modifications needed in the C$^*$-setting have been carried out in the recent paper by Jordans~\cite{MR3266525}. Their result states that any such C$^*$-tensor category is obtained by twisting $\Rep \SU_q(n)$ by a $3$-cocycle on its chain group (which is naturally isomorphic to the dual of the center of $\SU(n)$) for a uniquely determined $q\in(0,1]$. It should be remarked that there exist only partial results extending this classification to other compact connected simple Lie groups~\cite{MR2132671}, and this is essentially the only reason which does not allow us to extend our results to all such groups in place of~$\SU(n)$.

As for the second problem, the classification of monoidal autoequivalences of $\Rep\SU_q(n)$, and more generally of $\Rep G_q$ for any compact connected semisimple Lie group $G$, was obtained in~\citelist{\cite{MR2782190}\cite{MR2959039}}. From this it is easy to deduce a similar result for the $3$-cocycle twists of $\Rep \SU_q(n)$. The situation for twists of $\Rep G_q$ is slightly more complicated, exactly because we do not have a complete classification of categories with fusion rules of $G$. Nevertheless, with a bit more effort we can obtain a complete classification of monoidal autoequivalences for simple and simply connected~$G$.

Therefore the main remaining problem is classification of unitary fiber functors on the twists of $\Rep\SU_q(n)$ (or more generally of $\Rep G_q$), which induce the classical dimension function on the representation ring. In an earlier paper~\cite{MR3340190} we already constructed some fiber functors and studied the corresponding compact quantum groups $G^\tau_q$. In the present paper we show that for $q\ne1$, which exactly corresponds to the non-Kac case, all dimension-preserving unitary fiber functors on $\Rep G^\tau_q$ factor through $\Rep T$ in an essentially unique way, where $T$ is the maximal torus. This is deduced from a general result on dual $2$-cohomology of coamenable compact quantum groups obtained using a universal property of categorical Poisson boundaries we established in~\cite{poisson-bdry-monoidal-cat}. Finally, let us note that currently this last part is the only one place where it is essential to work within the operator algebraic framework, as the proofs in~\cite{poisson-bdry-monoidal-cat} heavily utilize the techniques from operator theory and subfactor theory.

\bigskip

\section{Preliminaries}

\subsection{Monoidal categories}

In this paper we deal with nonstrict C$^*$-tensor categories, following the conventions of~\cite{neshveyev-tuset-book}. For the convenience of the reader we briefly recall the basic terminology and notation.

A $\C$-linear category $\CC$ is a \emph{C$^*$-category} if its morphism sets $\CC(X, Y)$ are Banach spaces endowed with a conjugate linear and isometric involution $\CC(X, Y) \to \CC(Y, X)$, $T \mapsto T^*$, satisfying
\begin{align*}
(S T)^* &= T^* S^*,&
\| T^* T \| &= \| T \|^2.
\end{align*}
We always assume that such categories are closed under direct sums and subobjects.

A \emph{C$^*$-tensor category} is a C$^*$-category together with a bifunctor $\otimes\colon \CC \times \CC \to \CC$ and a distinguished object $\un = \un_\CC \in \Obj \CC$, together with natural unitary isomorphisms
\begin{align*}
\un \otimes X &\to X \leftarrow X \otimes \un,&
\Phi\colon (X \otimes Y) \otimes Z &\to X \otimes (Y \otimes Z)
\end{align*}
satisfying a number of conditions. We will always assume that the unit object is simple, meaning that $\CC(\un, \un) \simeq \C$.

A \emph{unitary tensor functor} between two C$^*$-tensor categories is given by a triple $F = (F_0, F_1, F_2)$ of the following form: $F_1$ is a linear functor $\CC \to \DD$ satisfying $F_1(T^*) = F_1(T)^*$, $F_0$ is a unitary isomorphism $\un \to F_1(\un)$, and $F_2$ is a natural unitary isomorphism $F_1(X) \otimes F_1(Y) \to F_1(X \otimes Y)$, satisfying the following compatibility condition for the associativity morphisms:
\begin{equation}\label{eq:tensor-functor-assoc-compat}
\xymatrix{
(F_1(X) \otimes F_1(Y)) \otimes F_1(Z) \ar[d]_{\Phi^\DD} \ar[r]^{\ \ \ \ F_2 \otimes \iota} & F_1(X \otimes Y) \otimes F_1(Z) \ar[r]^{F_2} & F_1((X \otimes Y) \otimes Z) \ar[d]^{F_1(\Phi^\CC)}\\
F_1(X) \otimes (F_1(Y) \otimes F_1(Z)) \ar[r]_{\ \ \ \ \iota \otimes F_2} & F_1(X) \otimes F_1(Y \otimes Z) \ar[r]_{F_2} & F_1(X\otimes(Y\otimes Z))
},
\end{equation}
and another set of compatibility conditions for the tensor units $\un_\CC$ and $\un_\DD$ and the isomorphism~$F_0$. If there is no fear of confusion we also write $F$ in place of $F_1$.  If we are given two unitary tensor functors $F \colon \CC \to \DD$ and $F'\colon \DD \to \DD'$, their composition $F' F \colon \CC \to \DD'$ is given by the triple $(F'_0 F'_1(F_0), F'_1 F_1, F'_1(F_2) F'_2)$.

Given two unitary tensor functors $F$ and $F'$ from $\CC$ to $\DD$, a natural unitary monoidal transformation $\eta\colon F \to F'$ is given by a natural transformation $\eta_X \colon F_1(X) \to F'_1(X)$ of functors $F_1$ and $F'_1$ by unitary morphisms $(\eta_X)_{X \in \Obj \CC}$ which satisfy $F_0 = F'_0 \eta_\un$ and $\eta_{X \otimes Y} F_2 = F'_2 (\eta_X \otimes \eta_Y)$. Such $F$ and $F'$ are said to be \emph{naturally unitarily monoidally isomorphic} if there are natural unitary monoidal transformations $\eta\colon F \to F'$ and $\eta'\colon F' \to F$ such that $\eta'_X \eta_X = \id_{F_1(X)}$ and $\eta_X \eta'_X = \id_{F'_1(X)}$.  Moreover, a unitary tensor functor $F\colon \CC \to \DD$ is said to be a unitary monoidal equivalence if there is another unitary tensor functor $G \colon \DD \to \CC$ such that $G F$ and $F G$ are naturally unitarily monoidally isomorphic to the identity functors of $\CC$ and $\DD$, respectively. We denote the group of unitary monoidal autoequivalences of $\CC$, considered up to natural unitary monoidal isomorphisms, by~$\Aut^\otimes(\CC)$.

An important class of C$^*$-tensor categories is \emph{rigid} C$^*$-tensor categories, in which any object has a dual. Namely, an object $Y \in \Obj \CC$ is a dual of $X \in \Obj \CC$, denoted as $Y = \bar{X}$, if there are morphisms $R_X \in \CC(\un, \bar{X} \otimes X)$ and $\bar{R}_X \in \CC(\un, X \otimes \bar{X})$ which satisfy the \emph{conjugate equations}
\begin{align*}
(R_X^* \otimes \iota_{\bar{X}}) \Phi^{-1} (\iota_{\bar{X}} \otimes \bar{R}_X) &= \iota_{\bar{X}},&
(\bar{R}_X^* \otimes \iota_{X}) \Phi^{-1} (\iota_{X} \otimes R_X) &= \iota_{X},
\end{align*}
up to the structure morphisms of the unit. The \emph{intrinsic dimension} (also called the quantum dimension) of $X$ is defined as
$$
d^\CC(X) = \inf_{(R_X, \bar{R}_X)} \left\| R_X \right\| \lVert \bar{R}_X \rVert,
$$
where $(R_X, \bar{R}_X)$ runs through the solutions of the conjugate equations for $X$. We will often omit the superscript $\CC$. A choice of $(R_X, \bar{R}_X)$ such that $d(X)^{1/2}=\left\| R_X \right\|= \lVert \bar{R}_X \rVert$ is called a \emph{standard solution}.

\subsection{Compact quantum groups}

Next let us review a few standard facts about compact quantum groups.  See again~\cite{neshveyev-tuset-book} for the details.

Let $G$ be a compact quantum group, given by a unital cancellative C$^*$-bialgebra $(C(G), \Delta)$. A finite dimensional \emph{unitary representation} of $G$ is given by a unitary element $U \in B(H_U) \otimes C(G)$ for some finite dimensional Hilbert space $H_U$ (the underlying Hilbert space of $U$), which satisfies the equality $U_{1 2} U_{1 3} = (\iota \otimes \Delta)(U)$ in $B(H_U) \otimes B(H_U) \otimes C(G)$. Here, $U_{i j}$ is given by distributing the legs of~$U$ to the positions indicated by the matching numbers, so that $U_{1 3}$ equals $\sum_i x_i \otimes 1 \otimes y_i$ if $U$ is given by $\sum_i x_i \otimes y_i$. We denote by $\dim U$ the dimension of $H_U$ and call it the \emph{classical dimension} of~$U$.

The category of finite dimensional unitary representations of $G$, denoted by $\Rep G$, has a natural structure of a C$^*$-tensor category with the tensor product $U \circt V = U_{1 3} V_{2 3}$, so that  $H_{U \circt V}=H_U \otimes H_V$. Any object $U$ of $\Rep G$ has a dual object $\bar U$, realized on the conjugate space~$\bar{H}_U$. Woronowicz's Tannaka--Krein duality theorem says that any compact quantum group $G$ can be recovered from $\Rep G$, which has an irreducible unit and duality of objects, and the unitary tensor functor (the \emph{canonical fiber functor of $G$}) $\Rep G \to \Hilb_f, U \mapsto H_U$.

A crucial part of the proof of the Tannaka--Krein duality is the reconstruction of the \emph{regular algebra} $\C[G] \subset C(G)$, defined as the linear span of the elements $(\omega \otimes \iota)(U)$, where $\omega \in B(H_U)_*$ and~$U$ runs through the finite dimensional unitary representations of $G$. It is a Hopf $*$-algebra, with the coproduct defined by the restriction of $\Delta$. As  a linear space it can be identified with
\begin{equation}
\label{eq:reg-alg-pres}
\bigoplus_{[U] \in \Irr G} \bar{H}_U \otimes H_U \cong \bigoplus_{[U] \in \Irr G} B(H_U)_*.
\end{equation}
The projection $h$ onto the direct summand $\bar{H}_\un \otimes H_\un \simeq \C$ extends to a unique invariant state on~$C(G)$, called the \emph{Haar state}. The C$^*$-algebra generated by $C(G)$ in the GNS-representation associated with~$h$, called the \emph{reduced form} of $G$, is denoted by $C^r(G)$. On the other hand, the universal C$^*$-algebraic envelope of $\C[G]$ is denoted by $C^f(G)$.

There is a unique element $\rho=f_1\in\C[G]^*$, called the \emph{Woronowicz character}, such that for any finite dimensional unitary representation $U$ of $G$ the operator $\rho_U=(\iota\otimes\rho)(U)\in B(H_U)$ is positive, $\Tr(\rho_U)=\Tr(\rho_U^{-1})$ and
$$
(\iota\otimes S^2)(U)=(\rho_U\otimes1)U(\rho_U^{-1}\otimes 1),
$$
where $S$ is the antipode on $\C[G]$. Then the dual unitary representation $\bar U$ can be defined~by
$$
\bar U=(j\otimes\iota)\big((\rho_U^{-1/2}\otimes1)U^*(\rho_U^{1/2}\otimes1)\big)\in B(\bar H_U)\otimes\C[G],
$$
where $j\colon B(H_U)\to B(\bar H_U)$ is the canonical $*$-anti-isomorphism given by $j(T)\bar\xi=\overline{T^*\xi}$.
Consider the maps $r_U\colon\C\to \bar H_U\otimes H_U$ and $\bar r_U\colon\C\to H_U\otimes\bar H_U$ defined by $r_U(1)=\sum_i\bar e_i\otimes e_i$ and $\bar r_U(1)=\sum_i e_i\otimes \bar e_i$, where $\{e_i\}_i$ is any choice of an orthonormal basis in~$H_U$ (the maps $r_U$ and $\bar{r}_U$ are independent of this choice).  Then, with $\bar U$ defined as above, as a standard solution of the conjugate equations for~$U$ we can take
$$
R_U=(\iota\otimes \rho_U^{-1/2})r_U\ \ \text{and}\ \ \bar R_U=(\rho_U^{1/2}\otimes\iota)\bar r_U.
$$

The element $\rho$ defines also a one-parameter group $(\tau_t)_{t\in\R}$ of Hopf $*$-algebra automorphisms of~$\C[G]$~by
$$
(\iota\otimes \tau_t)(U)=(\rho_U^{it}\otimes1)U(\rho_U^{-it}\otimes 1),
$$
which is called the \emph{scaling group}.

A compact quantum group $G$ is said to be \emph{coamenable} if the counit of $\C[G]$ extends to a bounded linear functional on $C^r(G)$. This is equivalent to $C^f(G) = C^r(G)$, but for our purposes, the following characterization is more important.
Let~$\Gamma_U$ be the operator on $\ell^2 (\Irr( G))$ defined by
$$
\Irr(G)\ni[V] \mapsto \sum_{[W]\in\Irr(G)} \dim\big(\Hom_G(H_W, H_{V \circt U})\big) [W].
$$
Then $\|\Gamma_U\|\le\dim U$, and $G$ is coamenable if and only if the equality holds for all $U$.

Another important class of compact quantum groups is quantum groups of \emph{Kac type}, characterized by any of the following conditions: 1) the antipode $S$ of $\C[G]$ satisfies $S^2 = \id$, 2) $\dim U = d(U)$ for all $U$ in $\Rep G$, 3) $h$ is a tracial state, 4) $(\tau_t)_{t\in\R}$ is trivial.

\subsection{Maximal Kac quantum subgroup}

When $G$ is a compact quantum group, there exists a unique \emph{maximal Kac quantum subgroup} $K$ of $G$, which first appeared in work of So{\l}tan~\cite{MR2210362}. In other words, $K$ is a compact quantum group of Kac type, we have a surjective homomorphism $\C[G]\to\C[K]$ of Hopf $*$-algebras, and if $H$ is any other compact quantum group with the same properties, then $\C[G]\to\C[H]$ factors through $\C[G]\to\C[K]$.

Explicitly, the quantum group $K$ can be described as follows. The ideal $I\subset \C[G]$ generated by the elements $a-S^2(a)$ for all $a\in\C[G]$ is easily seen to be a Hopf $*$-ideal. Hence the quotient $\C[G]/I$ is a Hopf $*$-algebra with involutive antipode, so it defines a closed quantum subgroup $K$ of $G$ of Kac type.  Clearly, $K$ has the required maximality property. One also says that the Hopf $*$-algebra $\C[K]$ is the \emph{canonical Kac quotient} of~$\C[G]$.

Alternatively, $K$ can be described as follows, which is the original definition of So{\l}tan~\cite{MR2210362}. Let~$J$ be the intersection of the kernels of the GNS-representations of $C^f(G)$ defined by all tracial states. Then we can put $C(K)=C^f(G)/J$. In other words, $I=J\cap\C[G]$, so the ideal $I$ consists of the elements $a\in\C[G]$ such that $\tau(a^*a)=0$ for all tracial states $\tau$ on $\C[G]$.

If $G$ is coamenable, $K$ can be also found from the noncommutative Poisson boundary of $\hat G$~\cite{MR1916370} as $L^\infty(G/K) \simeq H^\infty(\hat{G}; \mu)$ for any ergodic probability measure $\mu$ on $\Irr(G)$~\citelist{\cite{MR2335776}\cite{MR3291643}}. One of our main observations, which will be exploited in Section~\ref{sec:fact-fib-funct}, is that this characterization of maximal Kac quantum subgroups manifests itself already at the categorical level.

\subsection{Cohomology of the discrete dual}

Deformation problems for compact quantum groups are controlled by cohomology theory of the dual quantum groups, which plays the central role in this paper. We again refer the reader to~\cite{neshveyev-tuset-book} for a more thorough discussion.

Denote by $\U(G)$ the dual space $\C[G]^*$ of $\C[G]$. It is a $*$-algebra, which is canonically isomorphic to  the algebraic direct product $\prod_{[U] \in \Irr G} B(H_U)$. Namely, a linear functional $\omega\in \C[G]^*$ defines operators $\pi_U(\omega) = (\iota\otimes\omega)(U)\in B(H_U)$, and by~\eqref{eq:reg-alg-pres} the information about $\omega$ is precisely given by the family $(\pi_U(\omega))_{[U]\in\Irr G}$. More generally, we put
$$
\U(G^k) =(\C[G]^{\otimes k})^*\cong \prod_{[U_1],\cdots, [U_k] \in \Irr G} B(H_{U_1}) \otimes \cdots \otimes B(H_{U_k}).
$$
We may interpret $\U(G^k)$ as the space of (possibly unbounded) $k$-point functions on the ``discrete dual'' quantum group $\hat{G}$. If $G$ is a genuine commutative compact group, this agrees with the usual notion of functions on the direct product of the Pontryagin dual group $\hat{G}$.

A \emph{$2$-cochain} on $\hat{G}$ is an invertible element $\E \in \U(G^2)$. It is said to be \emph{invariant} if it commutes with the image of the ``coproduct''$\hat{\Delta}\colon\U(G)\to\U(G^2)$ obtained by duality from the product on $\C[G]$. A $2$-cochain $T$ is said to be a \emph{$2$-cocycle} if it satisfies
$$
(\E \otimes 1) (\hat{\Delta} \otimes \iota)(\E) = (1 \otimes \E)(\iota \otimes \hat{\Delta})(\E).
$$
Invariant cocycles are also called \emph{lazy} in the algebraic literature.
If $c$ is an invertible element in the center of $\U(G)$ (an \emph{invariant $1$-cochain}), then $(c \otimes c) \hat{\Delta}(c^{-1})$ is a $2$-cocycle. Such cocycles are called invariant \emph{$2$-coboundaries}. The set of invariant $2$-cocycles form a group under multiplication, and the invariant $2$-coboundaries form a subgroup. The quotient is denoted by $H^2_G(\hat{G}; \C^\times)$, and called the \emph{invariant $2$-cohomology group} of $\hat G$. If we restrict to the unitary elements instead of invertible elements throughout, the corresponding group is denoted by $H^2_G(\hat{G}; \T)$.

If $\E$ is an invariant unitary $2$-cocycle, the multiplication by $\E^{-1}$ on $H_U \otimes H_V$ can be considered as a unitary endomorphism of $U \circt V$ in $\Rep G$. Such endomorphisms form a natural unitary transformation of the bifunctor $\circt$ into itself. The cocycle condition corresponds to the fact that this transformation is a monoidal autoequivalence of $\Rep G$. Up to natural unitary monoidal isomorphisms, any autoequivalence of $\Rep G$ fixing the irreducible classes can be obtained in this way. Moreover, the cohomology relation of cocycles corresponds to the natural unitary monoidal isomorphism of autoequivalences. Thus $H^2_G(\hat{G}; \T)$ can be considered as the subgroup of $\Aut^\otimes(\Rep G)$ consisting of autoequivalences that preserve the isomorphism classes of objects. Without the unitarity, $H^2_G(\hat{G}; \C^\times)$ corresponds to a subgroup of monoidal autoequivalences of $\Rep G$ as a tensor category over $\C$.

\smallskip

Let $\E$ be an arbitrary unitary $2$-cocycle on $\hat{G}$, invariant or not, and $F\colon\Rep G\to\Hilb_f$ be the canonical fiber functor.  Then the triple $F_\E = (\id_\C, U \mapsto H_U, \E^{-1})$ defines another unitary tensor functor $\Rep G \to \Hilb_f$. By Woronowicz's Tannaka--Krein duality $F_\E$ can be considered as the canonical fiber functor of another compact quantum group $G_\E$ satisfying $\Rep G = \Rep G_\E$. Concretely, $\U(G_\E)$ coincides with $\U(G)$ as a $*$-algebra, but is endowed with the modified coproduct $\hat{\Delta}_\E(T) = \E \hat{\Delta}(T) \E^{-1}$.  By duality, $\C[G_\E]$ is the same coalgebra as $\C[G]$, but has a modified $*$-algebra structure dual to $\hat{\Delta}_\E$. Up to natural unitary monoidal isomorphisms, the functors $F_\E$ exhaust all unitary tensor functors $F'\colon\Rep G\to\Hilb_f$ satisfying $\dim F'(U) = \dim H_U$. Moreover, if $T$ is a unitary element in $\U(G)$ (a \emph{unitary $1$-cochain}), then $\E_T = (T \otimes T) \E \hat{\Delta}(T^{-1})$ defines another unitary fiber functor which is naturally unitarily monoidally isomorphic to $F_\E$. Thus, the \emph{$2$-cohomology} $H^2(\hat{G}; \T)$, which is the quotient of the set of unitary $2$-cocycles by the cohomology relation $\E \sim \E_T$, gives a complete parametrization of such unitary fiber functors up to natural unitary monoidal isomorphisms. We also note that there is an action of the group $H^2_G(\hat{G}; \T)$ on the set $H^2(\hat{G}; \T)$ given by multiplication on the right. This corresponds to the restriction of the obvious right action of $\Aut^\otimes(\Rep G)$ on the natural unitary monoidal isomorphism classes of unitary fiber functors $\Rep G \to \Hilb_f$.

\smallskip

A \emph{$3$-cocycle} on $\hat{G}$ is an invertible element $\Phi \in \U(G^3)$ which satisfies
$$
(1 \otimes \Phi)(\iota \otimes \hat{\Delta} \otimes \iota)(\Phi) (\Phi \otimes 1) = (\iota \otimes \iota \otimes \hat{\Delta})(\Phi) (\hat{\Delta} \otimes \iota \otimes \iota)(\Phi).
$$
Again, such a cocycle is called invariant if it commutes with the image of $\hat{\Delta}^{(2)} = (\hat{\Delta} \otimes \iota) \hat{\Delta} = (\iota \otimes \hat{\Delta}) \hat{\Delta}$. Such cocycles are also called \emph{associators}.  If $\Phi$ is an invariant unitary $3$-cocycle and $\E$ is an invariant unitary $2$-cochain, the $3$-cochain
$$
\Phi_\E = (1 \otimes \E) (\iota \otimes \hat{\Delta})(\E) \Phi (\hat{\Delta} \otimes \iota)(\E^{-1}) (\E^{-1}\otimes1)
$$
is again an invariant unitary $3$-cocycle. Cocycles of the form $\Phi_\E$ are said to be cohomologous to~$\Phi$.  This defines an equivalence relation on the set of invariant unitary $3$-cocycles, and the quotient set is called the \emph{invariant unitary $3$-cohomology} $H^3_G(\hat{G}; \T)$.

If $\Phi$ is an invariant unitary $3$-cocycle, its action on $H_U \otimes H_V \otimes H_W$ can be considered as a new associativity morphism on the  C$^*$-category $\Rep G$ with bifunctor $\circt$. This gives us a new C$^*$-tensor category $(\Rep G, \Phi)$, which has the same data as $\Rep G$ except for the new associativity morphisms defined by the action of~$\Phi$. It is not clear whether the new category is automatically rigid, but this is at least the case if $\Phi$ acts as a scalar on $H_U \otimes H_V \otimes H_W$ for irreducible $U$, $V$ and $W$. If $\E$ is an invariant unitary $2$-cochain, the categories $(\Rep G, \Phi)$ and $(\Rep G, \Phi_\E)$ are naturally unitarily monoidally equivalent, by means of the unitary tensor functor $(\id_\un, \Id_{\Rep G}, \E^{-1})\colon (\Rep G, \Phi) \to (\Rep G, \Phi_\E)$. This way $H^3_G(\hat{G}; \T)$ gives a parametrization of the categories of the form $(\Rep G, \Phi)$ considered up to unitary monoidal equivalences that preserve the isomorphism classes of objects.

\subsection{Twisted \texorpdfstring{$q$}{q}-deformations of compact Lie groups}

Finally, let us recall from~\cite{MR3340190} how to construct new quantum groups whose representation categories are nontrivial twists of $\Rep G$. Let~$G$ be a simply connected semisimple compact Lie group, and $T$ be its maximal torus.  We denote the weight lattice and the root lattice by~$P$ and~$Q$ respectively, so that $P/Q$ is naturally isomorphic to the Pontryagin dual of the center of $G$.

Suppose that $c$ is a $\T$-valued $2$-cochain on the dual group $\hat{T} = P$, such that its coboundary $\partial c$ happens to be invariant under $Q$ in each variable. Then $\partial c$ can be considered a $3$-cocycle $\Phi^c$ on $P/Q = \widehat{Z(G)}$.  In turn, $\Phi^c$ can be considered as an invariant $3$-cocycle on $\hat{G}$, and hence an associator on $\Rep G$.  Concretely, for irreducible representations $U$, $V$ and $W$, the associator $\Phi^c(U, V, W)$
acts as the scalar $\Phi^c(\omega_U,\omega_V,\omega_W)$, where $\omega_U\in\widehat{Z(G)}$ is the central character of $U$.

Since $\Phi^c$ is the coboundary of $c$ over $\hat{T}$, we have a unitary fiber functor $F_c\colon(\Rep G,\Phi^c)\to\Hilb_f$, which is identical to the canonical fiber functor on $\Rep G$, except that the tensor structure is given~by
$$
H_U \otimes H_V \to H_{U \circt V},\ \ \xi \otimes \eta \mapsto c^*(\xi \otimes \eta).
$$
Analogous to the case of twisting by $2$-cocycles, this functor defines a new compact quantum group~$G^c$ such that $\Rep G^c$ is monoidally equivalent to $(\Rep G, \Phi^c)$. Explicitly, $\C[G^c]=\C[G]$ as coalgebras, while the new $*$-algebra structure is defined by duality from $(\U(G),c\Dhat(\cdot)c^*)$.

Notice that, since $c$ is not assumed to be invariant, the category $(\Rep G, \Phi^c)$ need not be monoidally equivalent to $\Rep G$. In fact, for simple and simply connected $G$, we later show that these categories are monoidally equivalent if and only if the class of $\partial c$ is trivial in $H^3(P/Q; \T)$.

More generally, for any $0 < q < \infty$, the finite dimensional admissible unitary representations of the quantized universal enveloping algebra $U_q(\mathfrak{g})$ define the $q$-deformed compact quantum group $G_q$ ($\mathfrak{g}$ is the complexified Lie algebra of $G$). This quantum group contains $T$ as a closed subgroup, and the image of $Z(G)$ is central in $\U(G_q)$. Thus, by the same construction we obtain a new quantum group $G_q^c$ such that $\U(G_q^c) = (\U(G_q), c \hat{\Delta}_q(\cdot) c^*)$.  Because $c$ is defined on $\hat{T}$, the coproduct of any element $a \in \U(T)$ computed in $\U(G_q^c)$ is the same as $\hat{\Delta}_q(a) = \hat{\Delta}(a)$.  In particular, $T$ is still a closed subgroup of $G_q^c$.

\bigskip

\section{Factorization of fiber functors through Kac quantum subgroups}
\label{sec:fact-fib-funct}

Let $G$ be a compact quantum group. We are interested in unitary tensor functors $F\colon\Rep G\to\A$ defining the classical dimension function on $\Rep G$, that is, such that $\dim U=d^\A(F(U))$ for all $U$. By~\cite{poisson-bdry-monoidal-cat}*{Theorem~4.1}, if $G$ is coamenable, then there exists a universal functor $\Pi\colon\Rep G\to\PP$ with this property. More precisely, to apply the results of~\cite{poisson-bdry-monoidal-cat} we in addition have to assume that $\Irr G$ is at most countable. This will be the case in our applications in the later sections. In the general case, we could consider quotient compact quantum groups of $G$ corresponding to countably generated full rigid monoidal subcategories of $\Rep G$, and then pass to the limit. We leave the details of this reduction to the countable case to the interested reader.

The universal functor $\Pi\colon\Rep G\to\PP$ was constructed as the Poisson boundary of $\Rep G$ with respect to an ergodic measure, but for our current purposes it is more instructive to understand that up to an isomorphism it can be described as follows, see~\cite{poisson-bdry-monoidal-cat}*{Section~4}. Assume that $\Rep G$ is a C$^*$-tensor subcategory of a strict C$^*$-tensor category $\A$ such that $\dim U=d^\A(U)$ for all objects $U$ in $\Rep G$. Choose a standard solution $(R_U,\bar R_U)$ of the conjugate equations for~$U$ in $\Rep G$. There exists a unique positive automorphism $a_U$ of the object~$U$ in~$\A$ such that
$$
(\iota\otimes a_U^{1/2})R_U\ \ \text{and}\ \ (a_U^{-1/2}\otimes\iota)\bar R_U
$$
form a standard solution of the conjugate equation for $U$ in $\A$. Then as $\PP$ we can take the C$^*$-tensor category obtained by the subobject completion (idempotent completion) of the monoidal subcategory of $\A$ generated by $\Rep G$ and the morphisms $a_U$ for all $U$, and as $\Pi\colon\Rep G\to\PP$ we can take the embedding functor.

\begin{theorem}
Let $G$ be a coamenable compact quantum group and $K$ be its maximal Kac quantum subgroup.  Then the forgetful functor $\Rep G\to\Rep K$ is a universal unitary tensor functor defining the classical dimension function on $\Rep G$.
\end{theorem}

\bp We will give two proofs. For the first proof, we consider $\Rep G$ as a subcategory of $\Hilb_f$ and use the description of the universal functor $\Pi\colon\Rep G\to\PP$ given above. Then the above characterization of the morphisms $a_U$ imply that they are exactly equal to $\rho_U$.

Let $\PP$ be the C$^*$-tensor category obtained by the subobject completion of the monoidal subcategory of $\Hilb_f$ generated by $\Rep G$ and the morphisms~$\rho_U$.  By Woronowicz's Tannaka--Krein duality, $\PP$ defines a closed quantum subgroup $H < G$. Then, on the one hand, since $\rho_U$ is a morphism in the new category $\PP = \Rep H$, the eigenspaces of $\rho_U$ are representations of $H$. On the other hand, the conjugation by~$\rho_U^{it}$ gives the action of the scaling group $\tau_t$ of on the direct summand $B(H_U)_* \subset \C[G]$.  Since the quotient map $\C[G]\to\C[H]$ intertwines the scaling groups, we see that the scaling group of $H$ must be trivial. Thus $H$ is of Kac type.

Next, if $H'$ is a closed quantum subgroup of $G$ of Kac type, then using again that the quotient map $\C[G]\to\C[H']$ intertwines the scaling groups, we see that $\Rep H'$ must contain the morphisms~$\rho_U$, and consequently $\Rep H \subset \Rep H'$.  Thus, $H$ is maximal among such quantum subgroups, so $H=K$.  This completes the first proof.

\smallskip

For the second proof, we argue more formally. Let $\Pi\colon\Rep G\to \PP$ be a universal functor. Since the forgetful functor $\Rep G\to\Hilb_f$ factors through $\Pi\colon\Rep G\to\PP$, by Woronowicz's Tannaka--Krein duality we can identify $\Pi\colon\Rep G\to\PP$ with the forgetful functor $\Rep G\to\Rep H$ for a closed quantum subgroup $H$ of $G$. Since the quantum dimension function on $\Rep H$ coincides with the classical dimension function, the quantum group $H$ is of Kac type. Hence, for any finite dimensional unitary representation $U$ of~$G$, we have
$\End_K(H_U)\subset\End_H(H_U)$. On the other hand, by assumption the forgetful functor $\Rep G\to\Rep K$ factors through $\Rep G\to\Rep H$, so that $\End_H(H_U)$ can be embedded into $\End_K(H_U)$. Hence $\End_K(H_U)=\End_H(H_U)$, and since this is true for all $U$, we conclude that $H=K$.
\ep

We will need only the following particular case of this result.

\begin{corollary}\label{cor:coamen-2-cohom-classification}
There is a bijective correspondence between the (natural unitary monoidal) isomorphism classes of unitary fiber functors $\Rep K\to\Hilb_f$ and the isomorphism classes of unitary fiber functors $F\colon\Rep G\to\Hilb_f$ such that $\dim U=\dim F(U)$. Namely, the correspondence maps a functor $\Rep K\to\Hilb_f$ into its composition with the forgetful functor $\Rep G\to\Rep K$.
\end{corollary}

In cohomological terms, this can be rephrased as follows.

\begin{corollary}\label{cor:coamen-2-cohom-classification-layman}
The natural map $H^2(\hat{K}; \T) \to H^2(\hat{G}; \T)$ induced by the inclusion $\U(K) \to \U(G)$ is a bijection.
\end{corollary}

We finish this section with the following simple observation.

\begin{proposition} \label{pmaxKac}
Let $G$ be a compact quantum group, $K$ be its maximal Kac quantum subgroup and $\E$ be a unitary $2$-cocycle on $\hat K$. Then $K_\E$ is the maximal Kac quantum subgroup of $G_\E$.
\end{proposition}

\bp
The quantum dimensions of the unitary representations of $K_\E$ are the same as those of $K$, since $\Rep K_\E= \Rep K$ and quantum dimension is an intrinsic notion to the representation category. The classical dimensions also remain the same by construction. Thus the quantum and classical dimensions on $\Rep K_\E$ coincide, hence $K_\E$ is of Kac type.

It follows that $K_\E$ is contained in the maximal Kac quantum subgroup of $G_\E$. Since $\E^*$ is a unitary $2$-cocycle on $\hat G_\E$ and $(G_\E)_{\E^*}=G$, by swapping the roles of $G$ and $G_\E$ we conclude that $K_\E$ must be maximal.
\ep

\bigskip

\section{Compact quantum groups of Lie type}
\label{sec:case-q-deformations}

In this section $G$ denotes a simply connected semisimple compact Lie group. Recall that if $q>0$ and $c$ is a $\T$-valued $2$-cochain on the dual $\hat{T} = P$ of the maximal torus $T\subset G$, such that its coboundary $\partial c$ defines a $3$-cocycle $\Phi^c$ on $\widehat{Z(G)} = P/Q$, then we can define a compact quantum group $G^c_q$.

\subsection{Dimension-preserving fiber functors}
\label{sec:max-torus-twist-q}

We want to apply the results of Section~\ref{sec:fact-fib-funct} to $G^c_q$. For this we have to find the maximal Kac quantum subgroups of these quantum groups.

As was shown by Tomatsu~\cite{MR2335776}*{Lemma~4.10}, for $q\ne1$ the maximal Kac quantum subgroup of $G_q$ is $T$.  We have the following generalization of this result.

\begin{theorem}
\label{thm:can-Kac-quot-G-q-tau}
For any $q>0$, $q\ne1$, and any $\T$-valued $2$-cochain $c$ on $P$ such that $\partial c$ descends to~$P/Q$, the maximal Kac quantum subgroup of $G_q^c$ is the maximal torus $T$.
\end{theorem}

Before going to the proof we need to establish a simple lemma. It is well-known that the Woronowicz character $\rho=f_1$ of $G_q$ lies in $\U(T)\subset\U(G_q)$, see for example~\cite{neshveyev-tuset-book}*{Proposition~2.4.10}. Namely, it equals $q^{-2\rho^*}$, where $\rho^* \in \mathfrak{h}$ (the complexified Lie algebra of $T$) is the unique vector satisfying $\alpha(\rho^*) = d_\alpha=(\alpha,\alpha)/2$ for each simple root $\alpha$. In other words, under the isomorphism $\mathfrak{h}\cong\mathfrak{h}^*$ the element~$\rho^*$ corresponds to half the sum of positive roots. Recall also that by construction we have an identification $\U(G^c_q)=\U(G_q)$ as $*$-algebras.

\begin{lemma}\label{lem:wor-char-of-g-q-tau}
Under the identification of $\U(G_q^c)$ with $\U(G_q)$, the Woronowicz character~$f_1$ of~$G_q^c$ is again given by $q^{-2\rho^*}$.
\end{lemma}

\begin{proof}
Since $G^c_q$ does not change when we multiply~$c$ by a $2$-cochain living on $\widehat{Z(G)}$, without loss of generality we may assume that the cocycle $\Phi^c$ is normalized.

For an irreducible unitary representation $U$ of $G_q$, put $\lambda_U=\Phi^c(\omega_{\bar{U}},\omega_{ U},\omega_{ \bar{U}})$. Then in the category $(\Rep G_q, \Phi^c)$, the pair $(\lambda_U R_U, \bar{R}_U)$ solves the conjugate equations for $U$. When we apply the functor $F_c\colon (\Rep G_q,\Phi)\to\Hilb_f$, which defines $G^c_q$, this solution is transformed into $(\lambda_U c R_U, c \bar{R}_U)$.

Since we also have $c\in\U(T \times T)$, it follows that if we choose an orthonormal basis $\{\xi_i\}_i$ in $H_U$ consisting of weight vectors, then the vector $c R_U(1) \in \bar{H}_U\otimes H_U$ has the form $\sum_i\bar \xi_i\otimes \beta_i\rho^{-1/2}\xi_i$ for some $\beta_i\in\T$, where $\rho$ denotes the Woronowicz character of $G_q$. It follows that the Woronowicz character~$f_{-1/2}$ of~$G^c_q$  has the form $\rho^{-1/2}v$ for a unitary~$v$ commuting with~$\U(T)$. Hence $f_1=\rho$ and $v=1$.
\end{proof}

\begin{proof}[Proof of Theorem~\ref{thm:can-Kac-quot-G-q-tau}]
Observe that the ambiguity of the correspondence $c \mapsto [\Phi^c]$ is in that one may perturb $c$ by a $2$-cocycle on $\hat{T} = P$ or by a $2$-cochain on $\widehat{Z(G)}$. By Proposition~\ref{pmaxKac}, if Theorem~\ref{thm:can-Kac-quot-G-q-tau} is true for a cochain $c$ then it is true for any other cochain that differs from $c$ by a $\T$-valued $2$-cocycle on~$\hat T$.  Also, $G^c_q$ does not change when we multiply~$c$ by a $2$-cochain living on $\widehat{Z(G)}$.  It follows that for the proof of the theorem it suffices to consider particular cochains $c$ such that $[\Phi^c]$ exhaust the classes in $H^3(\widehat{Z(G)};\T)$ of the form $[\Phi^c]$. Such representatives $c_\tau$ were constructed in~\cite{MR3340190}*{Proposition~2.6}, and the corresponding quantum groups were denoted by $G_q^\tau$.

Let us recall the structure of the corresponding C$^*$-algebras~$C(G^\tau_q)$ as studied in~\cite{MR3340190}*{Section~3}. Given $\tau = (\tau_i)_{i=1}^{\rk G} \in Z(G)^{\rk G}$, we consider the finite subgroup $T_\tau$ of $T$ generated by the components~$\tau_i$. Then there is a group homomorphism $\psi\colon \hat{T}_\tau \to T/Z(G)$ characterized by $\langle \psi(\chi), \alpha_i \rangle = \tau_i$ for any $\chi \in \hat{T}_\tau$ and any simple root $\alpha_i$.  Composing this homomorphism with the conjugation action of $T/Z(G)$ on $C(G_q)$, we obtain an action of $\hat{T}_\tau$ on $\C[G_q]$, denoted by $\Ad \psi$.  The crossed product $C(G_q) \rtimes_{\Ad \psi} \hat{T}_\tau$ has an action of $T_\tau$ which is given by right translations on the copy of $C(G_q)$ and by the dual action on the copy of $\hat{T}_\tau$.  The C$^*$-algebra $(C(G_q) \rtimes_{\Ad \psi} \hat{T}_\tau)^{T_\tau}$ has the structure of a compact quantum group, induced by those of $G_q$ and $\hat{T}_\tau$.  This is our quantum group $G_q^\tau$.

Using the crossed product presentation of $C(G^\tau_q)$ and results of Soibelman on the representation theory of $C(G_q)$ for $q\ne1$, it is not difficult to obtain information on representations of the C$^*$-algebra~$C(G^\tau_q)$, see~\cite{MR3340190}*{Proposition~3.4}:
\begin{equation}
\label{eq:prim-spec-pres}
\Prim(C(G_q^\tau)) = \coprod_{w \in W} (\theta_w(\hat{T}_\tau) \backslash T / T_\tau) \times \widehat{\theta_w^{-1}(T_\tau)},
\end{equation}
where $W$ is the Weyl group and $\theta_w$ is a certain homomorphism of $\hat{T}_\tau$ into $T$ expressed in terms of~$\psi$ and~$w \in W$.

By Lemma~\ref{lem:wor-char-of-g-q-tau}, the scaling groups of $G_q^\tau$ and $G_q$ are both given by the conjugation action by $q^{-2it\rho^*}\in T$, $t\in\R$. We denote this common scaling group by $(\tau_t)_t$. We claim that the only irreducible representations of~$C(G_q^\tau)$ that are fixed, up to isomorphism, under $(\tau_t)_{t}$ are the evaluations at the points of $T$.

Since $\Ad \psi$ commutes with $\tau_t$, we obtain a natural extension of $\tau_t$ to $\C[G_q] \rtimes_{\Ad \psi} \hat{T}_\tau$ by letting it act trivially on the copy of $\hat{T}_\tau$. The embedding of $\C[G_q^\tau]$ into the crossed product is compatible with this extension of $\tau_t$.  The action of $T_\tau$ on $\C[G_q] \rtimes_{\Ad \psi} \hat{T}_\tau$, being implemented by the right torus translation action $\mathrm{rt}$ and the dual action, also commutes with the above extension of $\tau_t$.  Thus, up to a strong Morita equivalence, the action of the scaling group $(\tau_t)_t$ on $C(G_q^\tau)$ can be identified with the action of $\R$ on $C(G_q) \rtimes_{\mathrm{rt}} T_\tau \rtimes_{\Ad \psi, \widehat{\mathrm{rt}}} \hat{T}_\tau$ such that $\R$ acts by $(\tau_t)_t$ on $C(G_q)$ and trivially on the rest.

The description \eqref{eq:prim-spec-pres} of the primitive spectrum of $C(G_q)$ is obtained from the Morita equivalence $C(G^\tau_q)\sim C(G_q) \rtimes_{\mathrm{rt}} T_\tau \rtimes_{\Ad \psi, \widehat{\mathrm{rt}}} \hat{T}_\tau$, the Effros--Hahn machinery for crossed product by finite groups, and the identification of $\Prim(C(G_q))$ with $\coprod_{w \in W}T$. It follows from the above considerations that on the part of the spectrum of $C(G^\tau_q)$ labeled by $w \in W$ in~\eqref{eq:prim-spec-pres}, the scaling group induces the translation by $q^{-2it(w \rho^* - \rho^*)}$, $t\in\R$, on $\theta_w(\hat{T}_\tau) \backslash T / T_\tau$ (see~\citelist{\cite{MR2914062}*{Lemma~3.4}\cite{MR3009718}*{Lemma~8}} for the action of $T$ on $\Prim(C(G_q))$ induced by the translations on~$C(G_q)$).  Since $w \rho^* \neq \rho^*$ unless $w = e$, we obtain the claim.

The rest of the argument is identical to the proof of \cite{MR2335776}*{Lemma~4.10}. Namely, let $K$ be the maximal Kac quantum subgroup of $G_q^\tau$, and $\pi$ be an irreducible representation of $C(K)$.  Composing~$\pi$ with the restriction map $C(G_q^\tau) \rightarrow C(K)$, we obtain an irreducible representation of~$C(G_q^\tau)$, again denoted by $\pi$.  Since the restriction map $C(G^\tau_q)\to C(K)$ intertwines the scaling groups, and the scaling group of $K$ is trivial, $\pi$ has to be the evaluation at some point of $T$.  This proves that $K$ is contained in~$T$, hence $K = T$.
\end{proof}

\begin{remark}
 After this paper was submitted for publication we found another proof of the above theorem, which is of a more categorical flavour and is based on considerations of the categorical Poisson boundary. It applies to a more general setting than the above and will appear in a forthcoming joint work with Julien Bichon.
\end{remark}

Combining this theorem with Corollary~\ref{cor:coamen-2-cohom-classification} we obtain the following result.

\begin{corollary}\label{ccocycleclass}
Let $G^c_q$ ($q>0$, $q\ne1$) be as in Theorem~\ref{thm:can-Kac-quot-G-q-tau}. Then any dimension-preserving unitary fiber functor $\Rep G^c_q\to\Hilb_f$ factors, uniquely up to isomorphism, through $\Rep T$. Equivalently, the inclusion map $\U(T)\to\U(G^c_q)$ defines a bijection $H^2(\hat T; \T)\cong H^2(\hat G^c_q;\T)$.
\end{corollary}


\begin{remark} \label{rq1failure}
The above corollary is not true for $q=1$. The map $H^2(\hat T; \T)\to H^2(\hat G;\T)$ is neither surjective, if~$G$ is sufficiently large as discussed in the introduction, nor injective, as it factors through the quotient of $H^2(\hat T;\T)$ by the action of the Weyl group.
\end{remark}

\subsection{Monoidal equivalences}\label{sec:autoequiv}

If $c'$ is a $\T$-valued $2$-cocycle on $\hat{T} = P$ and $c$ is a $2$-cochain as in the previous section, the twisting of $G^c_q$ by $c'$ is equal to $G^{c c'}_q$, as seen from the coproduct formula for~$\U(G^c_q)$. Thus, we know from Corollary~\ref{ccocycleclass} that the class of compact quantum groups of the form~$G_q^c$ is closed under cocycle twisting when $q \neq 1$. Our goal is to classify such quantum groups up to isomorphism, which will be completed only for the simple case in this paper.

In general, a quantum group isomorphism $G \to G'$ is the same thing as a unitary monoidal equivalence $F \colon \Rep G' \to \Rep G$ and a natural unitary monoidal isomorphism $\eta\colon F^G F \simeq F^{G'}$, where $F^G$ and $F^{G'}$ are the canonical fiber functors of $G$ and $G'$. Therefore, in view of injectivity of the  map $H^2(\hat T; \T)\to H^2(\hat G^c_q;\T)$, in order to classify the quantum groups~$G^c_q$ up to isomorphism it remains to understand the groups $\Aut^\otimes(\Rep G^c_q)$ and classify the categories $\Rep G^c_q$ up to unitary monoidal equivalence.

\smallskip

We start by computing the invariant $2$-cohomology. Let us remark that our arguments in this section will work also for $q=1$.

\begin{proposition} \label{pinvcoh}
For any $q>0$ and any $\T$-valued $2$-cochain $c$ on $P$ such that $\partial c$ descends to $P/Q$, we have group isomorphisms
$$
H^2(P/Q; \T)\cong H^2_{G_q^c}(\hat{G}_q^c; \T)\cong H^2_{G_q^c}(\hat{G}_q^c; \C^\times),
$$
induced by the inclusion $\U(Z(G)) \to \U(G_q^c)$.
\end{proposition}

\bp
For the trivial $c$, the assertion is already proved in \cite{MR2959039}, see also~\cite{neshveyev-tuset-book}. The general case can be reduced to this. Let us again give two arguments for this, one categorical and the other Hopf algebraic.

Let $E$ be a unitary monoidal autoequivalence of $(\Rep G_q, \Phi^c)$ which maps every object to an isomorphic object.  Since for any simple objects $X$, $Y$ and $Z$ in $\Rep G_q$ the associativity morphism $\Phi^c(X, Y, Z)$ is scalar, it is easy to see that $E$ can be also considered as a unitary monoidal autoequivalence $\tilde E$ of $\Rep G_q$. Then, by \cite{MR2959039}*{Theorem~1}, $\tilde E$ is naturally unitarily monoidally isomorphic to an autoequivalence coming from some $[c'] \in H^2(P/Q; \T)$.  In turn, any such isomorphism defines a natural unitary monoidal isomorphism between $E$ and the autoequivalence induced by $c'$ on $(\Rep G_q, \Phi^c)$. The same works for nonunitary autoequivalences, and we obtain the assertion.

\smallskip
The second proof goes as follows. Recall that the coproduct of $\U(G^c_q)$ is given by $c \hat{\Delta}_q(x) c^*$, where~$\hat{\Delta}_q$ is the coproduct on $\U(G_q)$. Thus, an element $\E\in\U(G^c_q\times G^c_q)$ is an invariant $2$-cochain on~$\hat G^c_q$ if and only if $c^*\E c$ is an invariant $2$-cochain on $\hat G_q$. Moreover, we claim that $\E$ is an invariant $2$-cocycle on $\hat{G}^c_q$ if and only if $c^* \E c$ is an invariant $2$-cocycle on $\hat{G}_q$, and so we can apply to $c^* \E c$ the results of~\cite{MR2959039}.

Let us verify our claim. Suppose that $\E$ is an invariant $2$-cocycle on $\hat G^c_q$, so
$$
(\E c \otimes 1)(\hat{\Delta}_q \otimes \iota)(\E) (c^* \otimes 1) = (1 \otimes \E c) (\iota \otimes \hat{\Delta}_q)(\E) (1 \otimes c^*).
$$
We then want to verify the corresponding condition for $c^* \E c$ on $\hat G_q$,
$$
(c^* \E c \otimes 1) (\hat{\Delta}_q \otimes \iota)(c^* \E c) = (1 \otimes c^* \E c) (\iota \otimes \hat{\Delta}_q)(c^* \E c).
$$
Using the invariance of $\E$ with respect to the coproduct of $\hat{G}^c_q$, the left hand side is seen to be equal to $y = (\hat{\Delta}\otimes \iota)(c^*)(c^*\otimes1) ( \E c \otimes 1) (\hat{\Delta}_q \otimes \iota)(\E )(\Dhat\otimes\iota)(c)$, while the right hand side becomes $z = (\iota \otimes \hat{\Delta})(c^*) (1 \otimes c^*)(1\otimes \E c) (\iota \otimes \hat{\Delta}_q)(\E)(\iota\otimes\Dhat) (c)$. (Note that $c \in \U(T^2)$ implies the equalities of the form $(\hat{\Delta}_q \otimes \iota)(c) = (\hat{\Delta} \otimes \iota)(c)$, and that the terms like $c \otimes 1$ and $(\iota \otimes \hat{\Delta})(c)$ commute with each other.) Then, using that $\partial(c) = (1 \otimes c) (\iota \otimes \hat{\Delta})(c) (\hat{\Delta} \otimes \iota)(c^*) (c^* \otimes 1)$, we obtain
$$
z^{-1} y = (\iota \otimes \hat{\Delta})(c^*) (\iota \otimes \hat{\Delta}_q)(\E)^{-1}(1 \otimes \E c)^{-1} \partial(c) (\E c \otimes 1) (\hat{\Delta}_q \otimes \iota)(\E)(\Dhat\otimes\iota)( c).
$$
Now, the assumption $\partial(c) \in \U(Z(G)^3)$ implies that we may move $\partial(c)$ in the middle to outside, then the $2$-cocycle condition for $\E$ implies that
$$
z^{-1} y =  \partial(c) (\iota \otimes \hat{\Delta})(c^*) (1 \otimes c^*) (c \otimes 1) (\hat{\Delta} \otimes \iota)(c) = 1.
$$
Therefore $y = z$, which is what we wanted to show. The reverse implication can be argued in the same way.
\ep

Let us denote the based root datum of $G$ by $\Psi$. 
There is a natural action of $\sigma \in \Aut(\Psi)$ on the quantized universal enveloping algebra $U_q(\mathfrak{g})$ by Hopf $*$-algebra automorphisms, which is given on the generators by $E_j \mapsto E_{\sigma j}$, $K_j\mapsto K_{\sigma j}$.  This induces an action of $\Aut(\Psi)$ on $(\U(G_q),\Dhat_q)$. The restriction of this action to $T$, hence to $Z(G)$, does not depend on $q$. For a given $c$, we denote by $\Aut(\Psi)_{c,q}$ the stabilizer subgroup of $\Aut(\Psi)$ for the image of $\partial c$ in $H^3_{G_q}(\hat{G}_q; \T)$.

\begin{theorem} \label{tautoeq}
For any $q>0$ and any $\T$-valued $2$-cochain $c$ on $P$ such that $\partial c$ descends to $P/Q$, we have a short exact sequence
$$
1\to H^2(P/Q; \T)\to \Aut^\otimes(\Rep G_q^c)\to \Aut(\Psi)_{c,q}\to1.
$$
\end{theorem}

\bp
Let $R^+(G)$ be the representation semiring of $G$.  By a result of McMullen~\cite{MR733774}, the natural map $\Aut(\Psi) \to \Aut(R^+(G))$ is an isomorphism.  Thus, if $F$ is an autoequivalence of $\Rep G^c_q$, the induced automorphism of $R^+(G_q^c) = R^+(G)$ can be considered as an element of $\Aut(\Psi)$.  This way we obtain a group homomorphism $\Aut^\otimes(\Rep G_q^c) \to \Aut(\Psi)$. The kernel of the homomorphism  $\Aut^\otimes(\Rep G_q^c) \to \Aut(\Psi)$ consists of the monoidal autoequivalences preserving the isomorphism classes of objects, so it is isomorphic to $H^2(P/Q;\T)$ by Proposition~\ref{pinvcoh}.

Let us next show that the image of $\Aut^\otimes(\Rep G_q^c) \to \Aut(\Psi)$ is contained in~$\Aut(\Psi)_{c,q}$. Consider an autoequivalence of $\Rep G^c_q$ and let $\sigma$ be its image in $\Aut(\Psi)$. The action of $\Aut(\Psi)$ on $\U(G_q)$ defines an isomorphism $G^c_q\cong G^{\sigma(c)}_q$, hence a unitary monoidal equivalence between $(\Rep G_q,\Phi^c)$ and $(\Rep G_q,\Phi^{\sigma(c)})$. Since $\sigma$ is defined by an autoequivalence of $(\Rep G_q,\Phi^c)$, it follows that there exists a unitary monoidal equivalence $F = (F_0, F_1, F_2)$ between $(\Rep G_q,\Phi^c)$ and $(\Rep G_q,\Phi^{\sigma(c)})$ which preserves the isomorphism classes of objects of $\Rep G_q$. Replacing $F$ by a naturally unitarily monoidally isomorphic functor, we may assume that $F_1$ is simply the identity functor on $\Rep G_q$. Then $F_2$ is given by the action of $\E^{-1}$ for an invariant unitary $2$-cochain $\E$ on~$\hat G_q$, and~\eqref{eq:tensor-functor-assoc-compat} reads as $\Phi^{\sigma(c)}=\Phi^c_\E$. Thus $\sigma$ preserves the cohomology class of $\Phi^c$ in $H^3_{G_q}(\hat G_q;\T)$.

Reversing the above argument, it also becomes clear that the map $\Aut^\otimes(\Rep G_q^c) \to \Aut(\Psi)_{c,q}$ is surjective.
\ep

At this point the precise form of $\Aut(\Psi)_{c,q}$ for general $G$, $c$ and $q$ is not clear to us.  However, for a class of simple $G$ containing $\SU(n)$, we have the following result.

\begin{proposition}\label{prop:auto-psi-act-triv}
If $G$ is a simply connected simple compact Lie group not of type D$_{2m}$, we have $\Aut^\otimes(\Rep G^c_q) \cong \Aut(\Psi)$ for any $c$ and $q>0$.
\end{proposition}

\bp
By assumption $P/Q$ is a cyclic group, which implies that $H^2(P/Q; \T)$ is trivial. Moreover, $\Aut(\Psi)$ is either of order $2$ (type A$_n$, D$_{2m+1}$, E$_6$) or trivial, and for the latter case we have nothing to do. Notice that the natural map $H^3(P/Q; \T) \to H^3_{G_q}(\hat{G}_q;\T)$ is equivariant with respect to the action of $\Aut(\Psi)$. We claim that even when $\Aut(\Psi)$ contains a nontrivial element $\sigma$, it acts trivially already on $H^3(P/Q; \T)$.

Let us first consider the case of type A$_{n-1}$. The element $\sigma$ induces the automorphism of the Hopf $*$-algebra $U_q(\su(n))$ characterized by $E_i\mapsto E_{n-i}$, $1\le i\le n-1$. On the center of $\SU_q(n)$, identified with $\Z/n\Z$, it is given by $a\mapsto -a$. The induced map at the level of cocycles is $\phi \mapsto \psi$, where $\psi(a,b,c)=\phi(-a,-b,-c)$. But these cocycles are cohomologous: if $\phi$ is given by $\phi(a, b, c) = \omega^{(\lfloor \frac{a+b}{n} \rfloor - \lfloor \frac{a}{n} \rfloor - \lfloor \frac{b}{n}\rfloor) c}$ for an $n$th root of unity $\omega$, we see that $\phi\psi^{-1}$ is the coboundary of the $2$-cochain $(a,b)\mapsto\omega^{-\left(\lfloor\frac{a}{n}\rfloor +\lfloor-\frac{a}{n}\rfloor\right)b}$ on $\Z/n\Z$.

It remains to consider the types D$_{2m+1}$ and E$_6$. We either have $Z(G)  \cong \Z/4\Z$ or $Z(G) \cong \Z/3\Z$. The action of $\sigma$ must be either $a \mapsto -a$ or trivial, and again everything goes in the same way as in the A$_n$ case.
\ep

\begin{remark}
Proposition~\ref{prop:auto-psi-act-triv} fails for the D$_{2 m}$ case in two ways. First, $H^2(P/Q; \T)$ is no longer trivial. Next, as follows from Proposition~\ref{prop:simple-G-3-cohom-inj} below, $\Aut(\Psi)_{c,q}$ can be a proper subgroup of $\Aut(\Psi)$, depending on the choice of $c$ (but still independent of $q$).
\end{remark}

If the action of $\Aut(\Psi)$ on $H^3(P/Q;\T)$ is nontrivial, then computation of $\Aut(\Psi)_{c,q}$ is part of the more general task of finding when $\Phi^{c_1}$ and $\Phi^{c_2}$ define the same class in $H^3_{G_q}(\hat G_q;\T)$.

When $G$ is simple, the cocycles $\Phi^c$ exhaust the whole group $H^3(P/Q; \T)$, as observed in~\cite{MR3340190}*{Proposition~2.6}. Therefore in this case what we want is to understand the map $H^3(P/Q; \T) \to H^3_{G_q}(\hat{G}_q)$. We have the following result.

\begin{proposition}\label{prop:simple-G-3-cohom-inj}
If $G$ is a simply connected simple compact Lie group, then the
natural map $H^3(P/Q; \T) \to H^3_{G_q}(\hat{G}_q)$ is injective for
any $q>0$.
\end{proposition}

\bp For $G = \SU(n)$, the assertion already follows from the
Kazhdan--Wenzl theorem.  We follow their scheme also for the other
cases, cf.~\cite{MR1237835}*{Section~4}. The basic idea is to extract a numerical invariant of $(\Rep G_q, \Phi^c)$ using morphisms determined up to scalar multiples by the fusion rules and their compositions, in a way such that the ambiguity of the scalar factors is absorbed by the symmetry of the overall formula.

Let us first consider the cases other than the type~D$_{n}$. In these
cases we know that $P/Q$ can be identified with $\Z/k \Z$ for some
$k$.  Thus, $\Rep G$ is graded over $\Z/k \Z$. Moreover, the group
$H^3(P/Q; \T)$ is also cyclic and we may enumerate its elements by the
cocycles
$\Phi^j(a, b, c) = \omega_j^{(\lfloor \frac{a+b}{k} \rfloor - \lfloor
  \frac{a}{k} \rfloor - \lfloor \frac{b}{k}\rfloor) c}$, with
$\omega_j = e^{\frac{2 \pi \sqrt{-1} j}{k}}$ and $j = 0, 1, \ldots, k-1$.

Let us make the convention that for given $U\in \Rep G_q$, $c$, and
$k > 1$, the notation $U^{\otimes k}$ stands for the power of $U$
computed in the category $(\Rep G_q, \Phi^c)$ defined inductively as
$U^{\otimes k} = U \circt U^{\otimes (k-1)}$, so that one has
$U^{\otimes k} = U\circt (U \circt \cdots (U \circt U)\cdots)$. We
want to show that there is a way to choose $U \in \Rep G_q$ in the
homogeneous component of $[1] \in \Z/k\Z$ such that $U^{\otimes k}$
contains a unique copy of~$\un$, and such that for an isometric
morphism $f\colon \un \to U^{\otimes k}$ in $(\Rep G_q, \Phi^c)$,
unique up to a phase factor, the composition
\begin{equation}\label{eq:compos-f-assoc-f-star}
\xymatrix@C=2.5em{
U \simeq \un \circt U \ar[r]^-{f \otimes \iota} & (U \circt U^{\otimes( k-1)}) \circt U \ar[rr]^{\Phi^j_{U,U^{\otimes( k-1)},U}}&& U \circt (U^{\otimes ( k-1)} \circt U) \ar[r]^-{\iota \otimes f^*} & U \circt \un \simeq U
}
\end{equation}
in the category $(\Rep G_q, \Phi^j)$ is equal to
$\omega_j a_q \iota_U$ for some nonzero number $a_q$ which only
depends on~$q$. We note that there is an implicit use of the
associator to go from $U \circt (U^{\otimes ( k-1)} \circt U)$ to
$U \circt U^{\otimes k}$ in order for $\iota \otimes f^*$ to be
applicable. However, since we have $\Phi^j(a, b, c) = 1$ for
$0 \le a, b, a + b < k$, that associator acts trivially. By definition
of $\Phi^j$ we also see that $\Phi^j_{U,U^{\otimes ( k-1)},U}$ simply
contributes by the factor $\omega_j$. By rigidity of $\Rep G_q$ we
have $(\iota\otimes f^*)(f\otimes\iota)\ne0$ in $\Rep G_q$, once the
multiplicity of~$\un$ in $U^{\otimes k}$ equals one. Thus, all we have
to show is that there exists $U \in \Rep G_q$ in the homogeneous
component of $[1] \in \Z/k\Z$ such that $U^{\otimes k}$ contains a
unique copy of~$\un$.

\smallskip

In the following argument we refer to standard texts such
as~\cite{MR1153249} for the details on the representation theory of
simple compact Lie groups.

\smallskip

For $k = 2$, we want to find a self-conjugate irreducible
representation~$U$ in the homogeneous component of $[1] \in \Z/2\Z$ in
$\Rep G$.  For the type B$_n$ ($G = \Spin(2n+1)$) we can take the spin
representation as $U$. For the type C$_n$ ($G = \Sp(n)$), we can take
the defining representation of $\mathfrak{sp}(2n, \C)$ as $U$.  For
the type E$_7$, the unique irreducible representation of
dimension~$56$ is self-conjugate.

\smallskip

For $k = 3$, we only need to consider the type E$_6$ group. There are
two irreducible representations of dimension $27$, $U_\alpha$ and
$U_\beta$, which are conjugate to each other.  Moreover, the
representation $U_\alpha \circt U_\alpha$ contains $U_\beta$ with
multiplicity one (and the complement is spanned by two inequivalent
irreducible representations of dimension $351$). Hence the
multiplicity of $\un$ in $U_\alpha^{\otimes 3}$ equals one.

\smallskip

Now, let us consider the type D$_{n}$ case for $n$ odd.  In this case
$P/Q$ is still cyclic of order $4$, but there is no irreducible
representation $U$ in the homogeneous component of $[1] \in \Z/4\Z$
such that $U^{\otimes 4}$ contains $\un$ with multiplicity
one. However, we will see that there is still a way to specify $f
\colon \un \to U^{\otimes 4}$ only using the fusion rules, such that
the composition~\eqref{eq:compos-f-assoc-f-star} is nonzero.

Let $U_\pm$ denote the half-spin representations. In this case, $U_+$
is in the homogeneous component of $[1] \in \Z/4\Z$.  Since the
intertwiners $\un \to U_+^{\otimes 4}$ are not unique, we make the
following choice. In terms of the defining representation
$V \simeq \C^{2 n}$ (the \emph{vector representation}) of
$\mathfrak{so}(2 n, \C)$, we have a decomposition into irreducible
representations
\begin{equation}\label{eq:n-odd-half-spin-sq-decomp}
U_+ \circt U_+ \simeq U_{++} \oplus \wedge^{n-2} V \oplus \wedge^{n-4} V \oplus \cdots \oplus V
\end{equation}
in $\Rep G_q$, with an additional representation $U_{++}$.
Since the representation ring of $G_q$ is the same as that of $G$, we
continue to denote the corresponding irreducible representations in
$\Rep G_q$ by the same symbols $U_{++}$ and $\wedge^k V$.

Since $V$ is self-conjugate, it follows from \eqref{eq:n-odd-half-spin-sq-decomp}
that the multiplicity of $U_-$ in
$V\circt U_+$ equals one. Consider an isometry $f \colon \un \to
U_+^{\otimes 4}$ obtained as the composition of morphisms $\un \to
U_+\circt U_-\to U_+ \circt V \circt U_+$ and $U_+ \circt V \circt U_+
\to U_+^{\otimes 4}$, the latter induced by $V \to U_+^{\otimes
  2}$. This determines $f$ up to a phase factor. Then
Theorem~\ref{thm:spin-2n-n-odd-still-good} implies that  $(\iota
\otimes f^*) (f \otimes \iota) \neq 0$ in $\Rep G_q$.

\smallskip

It remains to consider the type D$_{n}$ case for $n$ even. The group
$G$ is again $\Spin(2 n)$, but $Z(\Spin(2 n))$ is isomorphic to
$\Z/2\Z \oplus \Z/2\Z$, consisting of $I_{2^n}$ in the Clifford
algebra $M_{2^n}(\C)$, $I_{2^{n-1}} \oplus(-I_{2^{n-1}})$ (the chiral
element), $-I_{2^n}$ (the nontrivial element in the kernel of
$\Spin(2n) \to \SO(2n)$), and $(-I_{2^{n-1}}) \oplus I_{2^{n-1}}$
(their product). In this case, we also know that the representations
$U_{\pm}$ are mutually inequivalent, irreducible, and self-conjugate,
and that the decomposition of $U_+ \circt U_-$ in $\Rep G$ into
irreducible representations is given by
$$
\wedge^{n-1} V \oplus \wedge^{n-3} V \oplus \cdots \oplus V.
$$

As in the odd case, we continue to use the same symbols for representations of $G_q$.
In terms of the grading of $\Rep G_q$ over
$P/Q \simeq \Z/2\Z \oplus \Z/2\Z$, $U_+$ has degree $(1, 0)$ while
$U_-$ has degree $(0, 1)$, and $V$ has $(1, 1)$. Moreover, the
generators of $H^3(P/Q; \T) \simeq (\Z/2)^{3}$ can be represented by
the following cocycles:
\begin{gather*}
\begin{align*}
\phi_1((a, a'), (b, b'), (c, c')) &= (-1)^{a b c},&
\phi_2((a, a'), (b, b'), (c, c')) &= (-1)^{a' b' c'},
\end{align*}\\
\phi_3((a, a'), (b, b'), (c, c')) = (-1)^{a b c'}.
\end{gather*}
Note that the first two are induced by the embeddings
$\Z/2\Z \to P/Q$. In terms of the central characters of $U_{\pm}$ and
$V$, denoting them by $\omega_{\pm}$ and $\omega_V$, we have
\begin{gather*}
\phi_1(\omega_+, \omega_+, \omega_+) = \phi_1(\omega_V, \omega_V, \omega_V) = -1, \quad \phi_1(\omega_-, \omega_-, \omega_-) = 1,\\
\phi_2(\omega_-, \omega_-, \omega_-) = \phi_2(\omega_V, \omega_V, \omega_V) = -1, \quad \phi_2(\omega_+, \omega_+, \omega_+) = 1,\\
\phi_3(\omega_+, \omega_+, \omega_+) = \phi_3(\omega_-, \omega_-, \omega_-) = 1, \quad \phi_3(\omega_V, \omega_V, \omega_V) = -1.
\end{gather*}
Now, take a category of the form $(\Rep G_q, \Phi^c)$.  First, using
each of $U_{\pm}$ as $U$ as in the above argument for with $k = 2$,
one may test whether $\partial c$ contains the classes of $\phi_1$ and
$\phi_2$. Next, using $V$ in that argument, one may test if
$\partial c$ contains the class of $\phi_3$. This way it is possible
to recover the cohomology class of $\partial c$ from the self-duality
morphisms of $U_{\pm}$ and $V$.  \ep

\begin{remark}
 It follows that $[\partial c]^2 \in H^3(P/Q;\T)$ is a cohomological obstruction to the existence of braiding on $(\Rep G_q; \Phi^c)$ for simple $G$, see~\cite{MR3340190}*{Remark~4.4}.
\end{remark}

Finally, let us make the following simple observation.

\begin{proposition} \label{pequalq} If $\Rep G^{c_1}_{q_1}$ and
  $\Rep G^{c_2}_{q_2}$ are unitarily monoidally equivalent for some
  $q_1,q_2>0$ and $\T$-valued $2$-cochains $c_1,c_2$ on $P$ such that
  $\partial c_1,\partial c_2$ descend to $P/Q$, then either
  $q_1=q^{-1}_2$ or $q_1=q_2$.
\end{proposition}

\bp It is well-known that there is an isomorphism between $G_q$ and $G_{q^{-1}}$ which maps the maximal torus into the maximal torus. It follows that $G^c_q\cong G^{c'}_{q^{-1}}$ for some $c'$. Therefore without loss of generality we may assume that $q_1,q_2\le1$. Next, a monoidal equivalence between $\Rep G^{c_1}_{q_1}$ and $\Rep G^{c_2}_{q_2}$ defines an automorphism of the representation semiring $R^+(G)$, so by the same argument as in the proof of Theorem~\ref{tautoeq} we conclude that for some $\sigma\in\Aut(\Psi)$, $\Rep G^{\sigma(c_1)}_{q_1}$ and $\Rep G^{c_2}_{q_2}$ are monoidally equivalent via an equivalence which defines the identity map on $R^+(G)$.

Now, the proposition follows by observing that the quantum dimension on $(\Rep G_q,\Phi^c)$ is independent of $c$, and if $U_1$ is any nontrivial irreducible representation of $G_{q_1}$ and $U_2$ is the corresponding representation of $G_{q_2}$ (via the identification of the representation rings), then $\dim_{q_1}U_1>\dim_{q_2}U_2$ if $q_1<q_2\le1$. Indeed, as we already reminded before Lemma~\ref{lem:wor-char-of-g-q-tau}, the quantum dimension on $\Rep G_q$ is defined by the trace of $q^{-2\rho^*}$, where $\rho^*\in\mathfrak h$ is the element corresponding to half the sum of positive roots. The element $\rho^*$ acts in every representation of $G$ by an operator with symmetric spectrum, since the longest element in the Weyl group maps $\rho^*$ into $-\rho^*$. Therefore our claim follows from the fact that the function $q\mapsto q+q^{-1}$ is strictly decreasing on the interval $(0,1]$.
\ep

\subsection{Isomorphisms}

Let us summarize what kind of isomorphisms between the quantum groups~$G^c_q$ we have for $q\le1$. First, if $b$ is a $2$-cochain defined on $P/Q$, the quantum group $G^{c b}_q$ is \emph{by definition} the same as $G^c_q$.  Second, if we replace $c$ by $c \partial e$ for some $1$-cochain $e$ on $P$, we have the isomorphism $G_q^c \simeq G^{c \partial e}_q$ implemented on the level of $\U(G_q)$ by conjugation by $e$, which can be regarded as a unitary in $\U(T)$.  Finally, any $\sigma \in \Aut(\Psi)$ induces an automorphism of $\U(G_q)$, which in turn induces an isomorphism $G^c_q\cong G^{\sigma(c)}_q$.

The following theorem states that for $q<1$ and simply connected simple groups this exhausts all the possibilities.

\begin{theorem}\label{thm:simple-G-isom-class}
Let $G$ be a simply connected simple compact Lie group, $q_1,q_2\in(0,1)$, and $c_1,c_2$ be $\T$-valued $2$-cochains  on $P$ such that $\partial c_1,\partial c_2$ descend to $P/Q$. Then the quantum groups~$G^{c_1}_{q_1}$ and~$G^{c_2}_{q_2}$ are isomorphic if and only if $q_1=q_2$ and there exist an element $\sigma \in \Aut(\Psi)$ and a $\T$-valued $2$-cochain $b$ on $P/Q$ such that $c_1 \sigma(c_2)^{-1}b^{-1}$ is a coboundary on $P$.
\end{theorem}

\bp
The ``if'' part is clear from the discussion preceding the theorem. In order to prove the ``only if'' part let us assume that $G^{c_1}_{q_1}\cong G^{c_2}_{q_2}$. By Proposition~\ref{pequalq} we have $q_1=q_2=q$. Next, the isomorphism defines an automorphism of the representation semiring of $G$, and consequently an element $\sigma^{-1}\in \Aut(\Psi)$. Since $G^{c_2}_q\cong G^{\sigma(c_2)}_q$, replacing $c_2$ by $\sigma(c_2)$ we may assume that $\sigma$ is trivial. Then the isomorphism $G^{c_1}_{q}\cong G^{c_2}_{q}$ gives us a unitary monoidal equivalence between $(\Rep G_q,\Phi^{c_1})$ and $(\Rep G_q,\Phi^{c_2})$ which preserves the isomorphism classes of objects of $\Rep G_q$. Hence $\Phi^{c_1}$ and $\Phi^{c_2}$ define the same class in $H^3_{G_q}(\hat G_q;\T)$. By Proposition~\ref{prop:simple-G-3-cohom-inj} we conclude that the $3$-cocycles $\partial c_1$ and $\partial c_2$ on $P/Q$ are cohomologous. Let $b$ be a $2$-cochain on $P/Q$ satisfying $(\partial c_1)( \partial c_2)^{-1} = \partial b$. Replacing $c_2$ by $c_2 b$ (which does not change $G^{c_2}_q$), we may further assume that $\partial c_1=\partial c_2$.

Let $F\colon\Rep G^{c_1}_q\to\Hilb_f$ be the unitary fiber functor corresponding to the dual cocycle $c_2c_1^{-1}$ on the maximal torus $T\subset G^{c_1}_q$.  Then $F$ defines the quantum group $G_q^{c_2}$, and our isomorphism $G_q^{c_1} \cong G_q^{c_2}$ defines an autoequivalence $E$ of $\Rep G^{c_1}_q$, which maps an object to an isomorphic object, such that $FE$ is naturally unitarily monoidally isomorphic to the canonical fiber functor $\Rep G^{c_1}_q\to\Hilb_f$. Then, Proposition~\ref{pinvcoh} implies that $E$ is naturally unitarily monoidally isomorphic to the functor defined by a $2$-cocycle $b'$ on $P/Q$ (trivial in all except the type D$_{2n}$ case). Replacing once again~$c_2$ by $c_2b'$ we may assume that $b'=1$. Then $F$ is naturally unitarily monoidally isomorphic to the canonical fiber functor $\Rep G^{c_1}_q\to\Hilb_f$. It follows that $c_2c_1^{-1}$, considered as a dual cocycle on~$G^{c_1}_q$, is a coboundary. Finally, by Corollary~\ref{ccocycleclass} we conclude that the cocycle $c_2c_1^{-1}$ on $P = \hat{T}$ is a coboundary.
\ep

\begin{remark} The theorem is not true for $q=1$, since as we already observed in Remark~\ref{rq1failure}, the map $H^2(\hat T;\T)\to H^2(\hat G;\T)$ is not injective.
\end{remark}

\bigskip

\section{Classification of non-Kac compact quantum groups of \texorpdfstring{$\SU(n)$}{SU(n)}-type up to isomorphism}
\label{sec:class-non-kac-su-n}

For $G=\SU(n)$ the results of the previous section can be further strengthened thanks to a classification theorem of Kazhdan and Wenzl~\cite{MR1237835}. It states that any semisimple rigid monoidal category with fusion rules of~$\SL(n)$ must be equivalent to one of the categories $(\Rep \SL_q(n), \Phi^c)$, where $q$ is not a nontrivial root of unity and $c$ is a $2$-cochain as in the previous section. The corresponding result for C$^*$-tensor categories states that any rigid C$^*$-tensor category with fusion rules of $\SU(n)$ is unitarily monoidally equivalent to $(\Rep \SU_q(n),\Phi^c)$ for some $q\in(0,1]$ and $c$, see~\cite{MR3266525} for details, as well as~\citelist{\cite{MR2307417}\cite{MR2825504}*{Section~7}} for related slightly weaker results.

\smallskip

Let us say that a compact quantum group $G$ is of $\SU(n)$-type if there is a dimension-preserving isomorphism of the representation semirings $R^+(G)\cong R^+(\SU(n))$. Then the fact that any rigid C$^*$-tensor category with fusion rules of $\SU(n)$ is unitarily monoidally equivalent to $\Rep \SU^c_q(n)$,  together with Corollary~\ref{ccocycleclass} implying that the class of quantum groups $\SU^c_q(n)$ is closed under cocycle twisting for $q\ne1$, show that any non-Kac compact quantum group of $\SU(n)$-type is isomorphic to $\SU^c_q(n)$. Combined with Theorem~\ref{thm:simple-G-isom-class} this gives a complete classification of such compact quantum groups up to isomorphism,  answering, in the non-Kac case, the question of Woronowicz raised at the end of~\cite{MR943923}.

To formulate the result more precisely, recall that for $\SU(n)$ the group $\Aut(\Psi)$ contains one nontrivial automorphism $\theta$. Its action on $\U(\SU_q(n))$, which already appeared in the proof of Proposition~\ref{prop:auto-psi-act-triv}, is characterized by $E_i\mapsto E_{n-i}$, $1\le i\le n-1$.  Restricted to the maximal torus $T$, this action can be expressed as
\begin{equation}\label{eq:theta-action}
\operatorname{diag}(t_1,\dots,t_n)\mapsto\operatorname{diag}(t_n^{-1},\dots,t_1^{-1}).
\end{equation}
Our main classification result can now be formulated as follows.

\begin{theorem}
\label{thm:answer-Woronowicz-non-Kac}
Let $G$ be a non-Kac compact quantum group of $\SU(n)$-type for some $n\ge2$. Then $G\cong\SU^c_q(n)$ for some $q\in(0,1)$ and a $\T$-valued $2$-cochain $c$ on $\hat{T}\cong\Z^{n-1}$ such that $\partial c$ descends to a $3$-cocycle on $\widehat{Z(\SU(n))}\cong\Z/n\Z$. Furthermore, two such quantum groups $\SU^{c_1}_{q_1}(n)$ and $\SU^{c_2}_{q_2}(n)$ are isomorphic if and only if $q_1=q_2$ and there exists a $\T$-valued $2$-cochain $b$ on $\widehat{Z(\SU(n))}$ such that either $c_1c_2^{-1}b^{-1}$ or $c_1\theta(c_2)^{-1}b^{-1}$ is a coboundary on $\hat T$.
\end{theorem}

Let us now find an explicit parameter set for the quantum groups $G$ of $\SU(n)$-type. Such a set can be obtained by the following procedure. Choose cochains $c_1,\dots,c_n$ such that $\partial c_i$ exhaust the group $H^3(\widehat{Z(\SU(n))};\T)\cong\Z/n\Z$.  Then $G\cong\SU^{c_k\omega}_q(n)$ for some $q\in(0,1)$, $1\le k\le n$ and a skew-symmetric bicharacter $\omega$ on $\hat T$. The numbers $q$ and $k$ are uniquely determined, and the above theorem tells us what the ambiguity in the choice of $\omega$ is.

In order to present $c_k$ and $\omega$ concretely, it is convenient to enlarge the maximal torus $T \simeq \T^{n-1}$ of $\SU(n)$ to that of $\rU(n)$, that is, to the group $\tilde T\cong\T^n$ of diagonal unitary matrices. Lifting $2$-cocycles on $\hat T$ to the group $\hat{\tilde T}\cong\Z^n$ with the base $(L_i)_{i=1}^n$ dual to the basis $(E_{ii})^n_{i=1}$ of the Lie algebra of diagonal matrices, we conclude that up to coboundaries such cocycles are represented by skew-symmetric bicharacters $\omega\colon \Z^n \times \Z^n \to \T$ which satisfy $\omega(L_1 + \cdots + L_n, x) = 1$ for any $x \in \Z^n$.  Putting $\omega_{i j} = \omega(L_i, L_j)\in\T$, the matrix $(\omega_{i j})_{i, j = 1}^n$ satisfies $\omega_{ii}=1$, $\omega_{j i} = \bar{\omega}_{i j}$ and $\prod_i \omega_{i j} = 1$ for any~$j$. Two such matrices $\omega=(\omega_{ij})_{i,j}$ and $\tilde\omega=(\tilde\omega_{ij})_{i,j}$ represent the same element of $H^2(\hat T;\T)$ if and only if $\omega^2_{ij}=\tilde\omega^2_{ij}$ for all~$i,j$.

As we already mentioned in the proof of Theorem~\ref{thm:can-Kac-quot-G-q-tau}, explicit examples of cochains $c$ were constructed in~\cite{MR3340190}. Namely, every $(n-1)$-tuple $\tau=(\tau_1,\dots,\tau_{n-1})$ of roots of unity of order $n$ defines a cochain $c_\tau$ such that
$$
c_\tau(\lambda,\mu+\alpha_i)=c_\tau(\lambda,\mu)\ \ \hbox{and}\ \  c_{\tau}(\lambda+\alpha_i,\mu)=\tau_i^{-|\mu|}
c_\tau(\lambda,\mu),
$$
where $\alpha_i=L_i-L_{i+1}$ are the simple roots and $|\cdot|\colon P\to\Z$ is defined by $|L_1|=n-1$, $|L_i|=-1$ for $2\le i\le n$. Such a cochain is not uniquely defined, but the ambiguity only contributes to a cochain on $P/Q$ and the corresponding quantum group $\SU^{\tau}_q(n) = \SU^{c_\tau}_q(n)$ is independent of any choices. Moreover, by~\cite{MR3340190}*{Proposition~4.1} we have $[\partial{c_\tau}]=[\partial{c_{\nu}}]$ in $H^3(P/Q;\T)$ if and only if
\begin{equation*}\label{eq:tau-i-nu-i-same-class}
\prod_{i=1}^{n-1}\tau_i^i=\prod^{n-1}_{i=1}\nu^i_i.
\end{equation*}
Therefore the required representatives $c_k$ can be obtained by taking, for example,
$$
\tau^{(k)} =(e^{2\pi (k-1) i/n},1,\dots,1).
$$

We next have to deform $\SU^{c_\tau}_q(n)$ by a skew-symmetric bicharacter $\omega$ on $\hat T$ to get a quantum group $\SU^{c_\tau\omega}_q(n)$, which we will also denote by $\SU^{\tau,\omega}_q(n)$. In terms of the quantum groups $\SU^{\tau,\omega}_q(n)$ Theorem~\ref{thm:answer-Woronowicz-non-Kac} can be reformulated as follows.

\begin{theorem} \label{thm:answer-Woronowicz-non-Kac2}
Let $G$ be a non-Kac compact quantum group of $\SU(n)$-type for some $n\ge2$. Then $G\cong\SU^{\tau,\omega}_q(n)$ for some $q\in(0,1)$, $\tau$ and $\omega$ as above. Furthermore, two such quantum groups $\SU^{\tau,\omega}_q(n)$ and $\SU^{\tau',\omega'}_{q'}(n)$ are isomorphic if and only if $q=q'$, $\prod^{n-1}_{i=1}\tau^i_i=\prod^{n-1}_{i=1}{\tau'_i}^{i}$ and one of the following holds for $\omega_{i j} = \omega(L_i, L_j)$ and $\omega'_{i j} = \omega'(L_i, L_j)$:
\begin{itemize}
\item[{\rm (i)}] $\displaystyle \omega_{ij}^2\prod^{j-1}_{k=i}\tau_k= {\omega'}^2_{ij}\prod^{j-1}_{k=i}\tau'_k$ for all $1\le i<j\le n-1$;

\item[{\rm (ii)}] $\displaystyle \omega_{ij}^2\prod^{j-1}_{k=i}\tau_k= {\omega'}^2_{n-i+1,n-j+1}\prod^{j-1}_{k=i}\bar\tau'_{n-k}$ for all $1\le i<j\le n-1$.
\end{itemize}
\end{theorem}

\bp We only have to show that the condition that there exists a cochain $b$ on $P/Q$ such that $c_\tau \omega (c_{\tau'}\omega')^{-1}b^{-1}$ (resp., $c_\tau \omega \theta(c_{\tau'}\omega')^{-1}b^{-1}$) is a coboundary, is equivalent to $\prod^{n-1}_{i=1}\tau^i_i=\prod^{n-1}_{i=1}{\tau'_i}^{i}$ and condition (i) (resp., condition (ii)).

Consider the first case. Let us first of all observe that the condition on the existence of $b$ is equivalent to saying that the $3$-cocycles $\partial(c_\tau\omega)$ and $\partial(c_{\tau'}\omega')$ on $P/Q$ are cohomologous and if
\begin{equation} \label{ecohom}
\partial(c_\tau\omega)\partial(c_{\tau'}\omega')^{-1}=\partial b
\end{equation}
holds for a $2$-cochain $b$ on $P/Q$, then $c_\tau \omega (c_{\tau'}\omega')^{-1}b^{-1}$ is a coboundary on $P$. This is an immediate consequence of the triviality of the group $H^2(P/Q;\T)$, as it implies that the cochain $b$ satisfying~\eqref{ecohom} is uniquely determined up to a coboundary on $P/Q$.

Now, as we already stated before the theorem, the condition that $\partial(c_\tau\omega)$ and $\partial(c_{\tau'}\omega')$ are cohomologous is equivalent to $\prod^{n-1}_{i=1}\tau^i_i=\prod^{n-1}_{i=1}{\tau'_i}^{i}$, since $\omega$ and $\omega'$ are cocycles. Once this condition is satisfied, we can define a bicharacter~$f$ on~$P$ by
$$
f(L_i,L_j)=\prod^{n-1}_{k=i}\tau_k\bar\tau'_k.
$$
It has the property that
\begin{align*}
f(\lambda,\mu+\alpha_i)&=f(\lambda,\mu),&
f(\lambda+\alpha_i,\mu)&=\tau_i^{-|\mu|}{\tau'_i}^{|\mu|} f(\lambda,\mu).
\end{align*}
It follows that $c_\tau c_{\tau'}^{-1}=f b$ for a $2$-cochain  $b$ on $P/Q$. Since $f$, $\omega$ and $\omega'$ are $2$-cocycles, $b$ satisfies~\eqref{ecohom}. Therefore the condition that $c_\tau \omega (c_{\tau'}\omega')^{-1}b^{-1}$ is a coboundary means that the cocycle $f\omega{\omega'}^{-1}$ is a coboundary. This is equivalent to saying that this cocycle is symmetric, and since it is a bicharacter and the bicharacters $\omega$ and $\omega'$ are skew-symmetric, this is equivalent to
$$
f(L_i,L_j)\omega({L_i,L_j})^2=f(L_j,L_i)\omega'(L_i,L_j)^2
\ \ \text{for}\ \ 1\le i<j\le n-1.
$$
This is exactly condition (i) in the formulation of the theorem.

Turning to the second case, we see from~\eqref{eq:theta-action} that $\theta(c_{\tau'}\omega')$ is equal to $c_{\tilde\tau'} \tilde{\omega}'$, given by $\tilde\tau'_i = \bar\tau'_{n-i}$ and $\tilde{\omega}'(L_i, L_j) = \omega'(L_{n-i+1},L_{n-j+1})$. We then proceed as in the first case.
\ep

Finally, let us present explicit generators and relations of $\C[\SU^{\tau,\omega}_q(n)]$.  The algebra $\C[\SU_q^\tau(n)]$ is generated by the matrix coefficients $v_{ij}$, $1\le i,j\le n$, of the canonical $n$-dimensional representation. Then $\C[\SU_q^{\tau,\omega}(n)]=\C[\SU_q^\tau(n)]$ as coalgebras, while the new product~$\cdot_\omega$ is determined by the following rule: if $x,y\in\C[\SU^\tau_q(n)]$ are such that
$$
(\pi\otimes\iota\otimes\pi)\Delta^{(2)}(x)=z_i\otimes x\otimes z_j\ \ \text{and}\ \ (\pi\otimes\iota\otimes\pi)\Delta^{(2)}(y)=z_k\otimes x\otimes z_l,
$$
where $\pi\colon\C[\SU^\tau_q(n)]\to\C[T]$ is the restriction map and we write $t=\operatorname{diag}(z_1(t),\dots,z_n(t))$ for elements $t\in T$, so that $\pi(v_{ij})=\delta_{ij}z_i$, then
$$
x\cdot_\omega y=\omega_{ik}\bar\omega_{jl}xy.
$$
From the relations in $\C[\SU^\tau_q(n)]$ given in~\cite{MR3340190}*{Section~4.3} we conclude that $\C[\SU^{\tau,\omega}_q]$ can be described as the universal algebra with generators $v_{ij}$, $1\le i,j\le n$, and relations
\begin{gather*}
\label{eq:SUqtN-rel-1}
  v_{i j} v_{i l} = \Bigl ( \prod_{j \le p < l} \tau_p^{-1} \Bigr ) q \bar\omega_{jl}^2 v_{i l} v_{i j} \quad (j < l), \quad  v_{i j} v_{k j} = \Bigl ( \prod_{i \le p < k} \tau_p \Bigr ) q
  {\omega}_{i k}^2v_{k j} v_{i j} \quad (i < k),\\
\label{eq:SUqtN-rel-2}
  v_{i j} v_{k l} = \Bigl ( \prod_{k\le p < i} \tau_p^{-1} \Bigr ) \Bigl ( \prod_{j \le p < l} \tau_p^{-1} \Bigr ) {\omega}_{i k}^2 \bar\omega_{j l}^2 v_{k l} v_{i j} \quad (i > k, j < l),\\
\label{eq:SUqtN-rel-3}
  \Bigl ( \prod_{j \le p < l} \tau_p \Bigr ) \omega_{jl}^2 v_{i j} v_{k l} - \Bigl ( \prod_{i \le p < k} \tau_p \Bigr ) \bar{\omega}_{ki}^2 v_{k l} v_{i j} = (q - q^{-1}) v_{i l} v_{k j} \quad (i < k, j < l),
  \end{gather*}
and
\[
\label{eq:SUqtN-rel-4}
  \sum_{\sigma \in S_n} \tau^{m(\sigma)} (-q)^{\absv{\sigma}} \bar{\omega}(1, \ldots, n) \omega(\sigma(1), \ldots, \sigma(n)) v_{1 \sigma(1)} \cdots v_{n \sigma(n)} = 1,
\]
where $m(\sigma)=(m(\sigma)_1,\dots,m(\sigma)_{n-1})$ is the multi-index given by $m(\sigma)_i = \sum_{k=2}^n (k - 1) m^{(k, \sigma(k))}_i$, with
$$
m^{(k,j)}_i=\begin{cases}1,&\text{if}\ \ k\le i<j,\\
-1,&\text{if}\ \ j\le i<k,\\
0,&\text{otherwise},\end{cases}
$$
and the function $\omega(i_1, \ldots, i_n)$ is defined by $\prod_{k < l} \omega_{i_k, i_l}$. The $*$-structure is uniquely determined by requiring the invertible matrix $(v_{ij})_{i,j}$ to be unitary.

\bigskip

\appendix
\section{Clebsch--Gordan coefficients of
  \texorpdfstring{$\Spin(2 n)$}{Spin(2n)}} \label{appendix}

Throughout this appendix $n\ge3$ is an odd integer. We review an explicit
decomposition of tensor powers of the half-spin representations of
$\Spin(2 n)$ (cf.~\cite{MR1153249}), and its extension to
$\Spin_q(2 n)$. Our goal is to prove the following theorem.

\begin{theorem}\label{thm:spin-2n-n-odd-still-good}
Let $q > 0$, and $f \colon \un \to U_+^{\otimes 4}$ be the unique
up to a phase factor isometric embedding of $\un$ into $U_+^{\otimes 4}$ in
$\Rep \Spin_q(2 n)$ which factors through
$U_+ \otimes V \otimes U_+$.  Then one has
$(\iota \otimes f^*) (f \otimes \iota) \neq 0$.
\end{theorem}

\subsection{Classical case}
\label{sec:classical-case}

Let us first consider the classical case $q = 1$.
Let $V$ be a complex vector space of dimension $2 n$ endowed with a
nondegenerate quadratic form $Q(\xi)$. The associated symmetric
bilinear form is denoted by
$(\xi, \eta)_Q = \frac{1}{2}(Q(\xi + \eta)-Q(\xi)-Q(\eta))$, so that
one has $Q(\xi) = (\xi, \xi)_Q$. The vector space $\wedge^2 V$ becomes
a Lie algebra with Lie bracket
$$
[\xi \wedge \eta, \xi' \wedge \eta'] = (\eta, \eta')_Q \xi \wedge \xi'
- (\eta, \xi')_Q \xi \wedge \eta' + (\xi, \xi')_Q \eta\wedge \eta' -
(\xi, \eta')_Q\eta\wedge\xi',
$$
which is isomorphic to $\mathfrak{so}(2 n, \C)$. It acts on $V$ by
$(\xi \wedge \eta) . \zeta = - (\eta, \zeta)_Q \xi + (\xi, \zeta)_Q
\eta$.
We denote the coproduct on the universal enveloping algebra by
$\Dhat$.

Let $W$ be a maximal isotropic subspace of $V$. We identify its dual
$W^*$ with a subspace of $V$ by means of~$Q$. Let $e_1, \ldots, e_n$
be a basis of $W$, and $(e_{n+1}, \ldots, e_{2 n})$ be the dual basis
in $W^* \subset V$. We represent the linear operators on $V$ using the
matrix units $E_{i, j}$ with respect to $(e_1, \ldots, e_{2 n})$. For
example, $e_{n+2} \wedge e_1$ is represented by
$E_{1, 2} - E_{n+2, n+1}$.

For a Cartan subalgebra $\mathfrak{h}$ of $\wedge^2 V$, we take the
linear span of the elements $e_{n+i} \wedge e_i = E_{i, i} - E_{n+i, n+i}$ for
$i = 1, \ldots, n$, and let $(L_i)_{i=1}^n \subset \mathfrak{h}^*$ be
the dual weights, so $(L_i, e_{n+j}\wedge e_j) = \delta_{i, j}$. A
standard choice of simple positive roots is $\alpha_i=L_i-L_{i+1}$ for $i=1,\dots,n-1$ and $\alpha_n=L_{n-1}+L_n$. The corresponding generators of the Lie algebra are $X_i = e_{n+i+1}\wedge e_i = E_{i, i+1} - E_{n+i+1, n+i}$ for
$i=1, \ldots n-1$ and $X_n = e_n \wedge e_{n-1} = E_{n-1,2n} -
E_{n,2n-1}$, $Y_i = e_{n+i}\wedge e_{i+1}=E_{i+1,i}-E_{n+i,n+i+1}$ for $i=1, \ldots, n-1$ and
$Y_n = e_{2 n} \wedge e_{2n-1}=E_{2n-1,n}-E_{2n,n-1}$. We note that $V$ has the highest weight
$L_1$, with the highest weight vector $e_1$.

The complex Clifford algebra $\Cliff(Q)$ associated with $Q$ is
generated by $c(\xi)$ for $\xi \in V$, subject to the relations
$$
c(\xi) c(\eta) + c(\eta) c(\xi) = -2(\xi, \eta)_Q = Q(\xi) + Q(\eta)
- Q(\xi + \eta).
$$
Considering $\Cliff(Q)$ as a Lie algebra with the commutator bracket, we have a Lie algebra embedding
$\wedge^2 V \to \Cliff(Q)$ given by
$\xi \wedge \eta \mapsto \frac{1}{4}[c(\xi), c(\eta)] = \frac{1}{2}
c(\xi) c(\eta) + \frac{1}{2} (\xi, \eta)_Q$.
Then the action of $\xi \wedge \eta \in \wedge^2 V$ on $V$ can be
identified with the adjoint action of $\frac{1}{4}[c(\xi), c(\eta)]$
on $c(\zeta)$ for $\zeta \in V$. In the following we put
$c(e_i) = c_i$ for $i = 1, \ldots, 2n$.

Put $S = \wedge^* W$, $U_+ = \wedge^{\odd} W$, and
$U_- = \wedge^{\even} W$. We have a representation of $\Cliff(Q)$ on
$S$ given by $c_i w = \sqrt{2} e_i \wedge w$ and
$c_{n+i} w = - \sqrt{2} e_i \lefthalfcup w$ for $i = 1,
\ldots,n$.
Then, the induced action of $\wedge^2 V$ preserves both $U_+$ and
$U_-$. If $X = \{x_1 < x_2 < \cdots < x_k \}$ is a subset of $\{1, \ldots, n\}$, we let~$e_X$ denote the vector $e_{x_1} \wedge \cdots \wedge
e_{x_k}$. Then $e_{n+i} \wedge e_i$ acts as the multiplication by
$-\frac{1}{2}$ on $e_X$ if $i$ does
not belong to $X$, and by $\frac{1}{2}$ if it does. In other words,
\begin{equation}\label{eweight}
\text{the vector } e_X\text{ has weight }\
\frac{1}{2}\sum_{i\in X}L_i-\frac{1}{2}\sum_{i\notin X}L_i.
\end{equation}
The representation $U_+$ has the highest weight
$\frac{1}{2}(L_1 + \cdots + L_n)$, while $U_-$ has the lowest weight
$-\frac{1}{2}(L_1+\cdots+L_n)$ (they are conjugate to each other). One
sees that the corresponding highest/lowest weight vectors are given by
$e_{\{1,\ldots,n\}}=e_1 \wedge \cdots \wedge e_n \in \wedge^n W
\subset U_+$ and $e_\emptyset=1 \in \wedge^0 W \subset U_-$.

Now, put $v_i = \frac{\sqrt{-1}}{\sqrt{2}}(e_i - e_{n+i})$ and
$v_{n+i} = \frac{1}{\sqrt{2}}(e_i + e_{n+i})$ for $i = 1, \ldots, n$,
and denote their span over $\R$ by $V_0$.  Then $Q$ restricts to a
positive definite bilinear form on $V_0$, with the orthonormal basis
$v_1, \ldots, v_{2n}$. Thus, $\wedge^2 V_0$ is a real Lie subalgebra
of $\wedge^2 V$ which is isomorphic to $\mathfrak{so}(2n,
\R)$.
Moreover, if we define a Hermitian inner product on $S$ so that the
vectors $e_X$ ($X \subset \{1, \ldots, n\}$) form an orthonormal
basis, the elements $c(v_i)$ ($i=1,\ldots,2n$) act as skew-adjoint
operators. Consequently, our model of $\Spin(2n)$ is the closed
connected subgroup of $\mathrm{GL}_1(\Cliff(Q))$ with Lie algebra
$\wedge^2 V_0$, and the half-spin representations are $U_\pm$ with the
orthonormal bases consisting of vectors~$e_X$ with odd/even~$|X|$. Note also that that we can define a $\wedge^2 V_0$-invariant Hermitian inner product on~$V$ by letting
$(e_i)_{i=1}^{2n}$ to be an orthonormal basis.

\smallskip

In order to find a copy of $V$ in $U_+ \otimes U_+$ (see the
decomposition~\eqref{eq:n-odd-half-spin-sq-decomp}), we need to find a
highest weight vector of weight $L_1$. For this purpose, given a
subset $X \subset \{1, \ldots, n\}$, let us define
$$
I_1(X)=\sum_{k\in X}k-|X|n\ \ \text{and}\ \ \sigma_1(X)=(-1)^{I_1(X)}.
$$
Let us also consider the set $\Omega_1$ of pairs of sets $(X',X'')$ such that $X'\cup X''=\{1,\dots,n\}$, $X'\cap X''=\{1\}$ and $|X'|$ is odd. Then, we claim that
\begin{equation*}
  \tilde{e}_1 = \sum_{(X',X'')\in\Omega_1}\sigma_1(X') e_{X'} \otimes e_{X''}
\end{equation*}
is a highest weight vector of weight $L_1$ in $U_+ \otimes
U_+$. That the weight is correct follows immediately from~\eqref{eweight}.
To see that $\tilde{e}_1$ belongs to
the kernel of the operators $X_i$, we compute $\Dhat(X_i)\tilde e_1$ as follows.
If $i=1$, then we clearly get zero, since $e_1\wedge e_X=0$ for any $X$ containing $1$.
Consider now $2\le i\le n-1$. The action of
$\Dhat(e_{n+i+1} \wedge e_i)$ is given by the sum of the actions by $-e_{i+1} \lefthalfcup e_i \wedge$ on the tensor components, so the image
of $\tilde{e}_1$ can be expressed as
$$
\sum_{\substack{(X',X'')\in\Omega_1\\ i\notin X',\ i+1\in X'}} \sigma_1(X') e_{X'\cup\{i\}\setminus\{i+1\}}\otimes e_{X''} + \sum_{\substack{(Y',Y'')\in\Omega_1\\ i\in Y',\ i+1\notin Y'}} \sigma_1(Y')e_{Y'}\otimes e_{Y''\cup\{i\}\setminus\{i+1\}}.
$$
We see that $X'$ in the first sum and $Y'=X'\cup\{i\}\setminus\{i+1\}$ in the second contribute with the same terms up to a sign. Since $I_1(X')=I_1(Y')+1$, these terms in fact cancel. Thus,
$\Dhat(e_{n+i} \wedge e_i) \tilde{e}_1 = 0$. As for the remaining element
$X_n = e_n \wedge e_{n-1}$, we compute
$\Dhat(X_n) \tilde{e}_1$ as
$$
-\sum_{\substack{(X',X'')\in\Omega_1\\ n-1,n\notin X'}} \sigma_1(X') e_{X'\cup\{n-1,n\}}\otimes e_{X''} - \sum_{\substack{(Y',Y'')\in\Omega_1\\ n-1,n\in Y'}} \sigma_1(Y')e_{Y'}\otimes e_{Y''\cup\{n-1,n\}}.
$$
This time, the term in the first sum corresponding to some $X'$ cancels the term in the second sum corresponding to $Y'=X'\cup\{n-1,n\}$, since $I_1(X')=I_1(Y')+1$.

Now, let us find the images $\tilde e_i$ of $e_i$ under the intertwiner
$V \to U_+ \otimes U_+$ mapping $e_1$ to $\tilde{e}_1$. This can be done by observing that $e_{i+1}=Y_ie_i$ and $e_{n+i}=-Y_ie_{n+i+1}$ for $1\le i\le n-1$, and $e_{2n}=-Y_ne_{n-1}$. In order to formulate the result, we need to introduce more notation. For $X\subset\{1,\dots,n\}$ and $1\le i\le n$ put
\begin{align*}
I_i(X)&=I_1(X)+|\{k\in X\mid k< i\}|-(i-1), & \sigma_i(X)&=(-1)^{I_i(X)},\\
I_{n+i}(X)&=I_1(X)+|\{k\in X\mid k<i\}|-(n-1), & \sigma_{n+i}(X)&=(-1)^{I_{n+i}(X)},
\end{align*}
and
\begin{align*}
\Omega_i&=\{(X',X'')\mid X'\cup X''=\{1,\dots,n\},\ X'\cap X''=\{i\},\ |X'|\text{ is odd}\},\\
\Omega_{n+i}&=\{(X',X'')\mid X'\cup X''=\{1,\dots,n\}\setminus\{i\},\ X'\cap X''=\emptyset,\ |X'|\text{ is odd}\}.
\end{align*}
Then
\begin{align}\label{eq:V-hwv-in-S-plus-squared}
\tilde{e}_i &= \sum_{(X',X'')\in\Omega_i}\sigma_i(X') e_{X'} \otimes e_{X''}, &
\tilde{e}_{n+i} &= (-1)^i\sum_{(X',X'')\in\Omega_{n+i}}\sigma_{n+i}(X') e_{X'} \otimes e_{X''}.
\end{align}
Let us check, for example, that $\tilde e_{2n}=-\Dhat(Y_n)\tilde e_{n-1}$. As $Y_n=e_{2n}\wedge e_{2n-1}$, the vector $-\Dhat(Y_n)\tilde e_{n-1}$ equals
$$
-\sum_{\substack{(X',X'')\in\Omega_{n-1}\\ n\in X'}} \sigma_{n-1}(X') e_{X'\setminus\{n-1,n\}}\otimes e_{X''} - \sum_{\substack{(Y',Y'')\in\Omega_{n-1}\\ n\notin Y'}} \sigma_{n-1}(Y')e_{Y'}\otimes e_{Y''\setminus\{n-1,n\}}.
$$
This expression equals $\tilde e_{2n}$, since for any $X'$ in the first sum we have $I_{n-1}(X')=I_{2n}(X'\setminus\{n-1,n\})$ and for any $Y'$ in the second sum we have $I_{n-1}(Y')=I_{2n}(Y')$.

\smallskip

Next, since $V$ is irreducible and self-conjugate, $U_+ \otimes V$
contains $U_-$ with multiplicity one. In order to find this inclusion
we need to find a lowest weight vector of weight
$-\frac{1}{2}(L_1+\cdots + L_n)$. It is given~by
\begin{equation} \label{eq:S-minus-in-S-plus-times-V}
e_1 \otimes e_{n+1} + \cdots + e_n \otimes e_{2n}.
\end{equation}
Again, the weight is correct by \eqref{eweight}, since the weight of $e_{n+i}\in V$ is $-L_i$. To see that we indeed get a lowest weight vector, we can compute the action of $Y_i=e_{n+i}\wedge e_{i+1}$ as
$$
- e_i \lefthalfcup e_{i+1} \wedge e_i \otimes e_{n+i} + e_{i+1} \otimes
(E_{i+1,i}-E_{n+i,n+i+1}) e_{n+i+1} = e_{i+1} \otimes e_{n+i} - e_{i+1} \otimes
e_{n+i} = 0,
$$
while the action of $Y_n=e_{2 n} \wedge e_{2 n - 1}$ clearly gives zero, since $e_i\in U_+$ and $e_{n+i}\in V$ are both annihilated by $Y_n$.

It follows that
the unique up to a scalar factor inclusion $g \colon U_- \to U_+^{\otimes 3}$ which factors through
$U_+ \otimes V$ is characterized by
$1 \mapsto \sum_{i=1}^n e_i \otimes \tilde{e}_{n+i}$. Similarly, we
can define another inclusion $h\colon U_- \to U_+^{\otimes 3}$ which
factors through $V \otimes U_+$, by
$1 \mapsto \sum_{i=1}^n \tilde{e}_{n+i} \otimes e_i$.

\begin{proposition}\label{prop:g-star-h-compos}
We have $g^* h = (-1)^{\frac{n+1}{2}} n (n-1)$.
\end{proposition}

\bp We just need to compute the inner product of the corresponding
lowest weight vectors $\sum_{i=1}^n e_i \otimes \tilde{e}_{n+i}$ and
$\sum_{i=1}^n \tilde{e}_{n+i} \otimes e_i$ inside $U_+^{\otimes
  3}$.
Putting $X_{i,j} = \{1,\ldots,n\}\setminus\{i,j\}$, we see that the
only contributing terms are
$$
\sum_{i \neq j} (-1)^{i}\sigma_{n+i}(X_{i,j}) e_i \otimes e_{X_{i,j}}
\otimes e_j
$$
from $\sum_{i=1}^n e_i\otimes \tilde{e}_{n+i}$, and
$$
\sum_{i \neq j} (-1)^j\sigma_{n+j}(\{i\}) e_i \otimes e_{X_{i,j}} \otimes
e_j
$$
from $\sum_{i=1}^n \tilde{e}_{n+j} \otimes e_j$.
If $i > j$, we have $I_{n+i}(\{j\})=-2n+j+2$ and $$I_{n+i}(X_{i,j})=-\frac{n(n-3)}{2}-j-1,$$
so that $\sigma_{n+i}(\{j\})=(-1)^j$ and $\sigma_{n+i}(X_{i,j})=(-1)^{\frac{n+1}{2}+j+1}$. On the other hand,
if $i < j$, then $I_{n+i}(\{j\})=-2n+j+1$ and $I_{n+i}(X_{i,j})=-\frac{n(n-3)}{2}-j,$
so that $\sigma_{n+i}(\{j\})=(-1)^{j+1}$ and $\sigma_{n+i}(X_{i,j})=(-1)^{\frac{n+1}{2}+j}$. Thus in any case
we have
\begin{equation}
  \label{eq:inn-prod-lwv}
 (-1)^{i+j} ( \sigma_{n+i}(X_{i,j}) e_i \otimes e_{X_{i,j}} \otimes e_j,
\sigma_{n+j}(\{i\}) e_i \otimes e_{X_{i,j}} \otimes e_j) =
(-1)^{\frac{n+1}{2}},
\end{equation}
and the assertion follows.
\ep

\bp[Proof of Theorem~\ref{thm:spin-2n-n-odd-still-good} for $q=1$] By the
uniqueness of the embedding of $\un$ inside $U_+ \otimes V \otimes U_+$, we
may also regard $f$ either as the composition of
$\bar{R}_+\colon \un \to U_+ \otimes U_-$ and
$\lambda \iota \otimes h$, or as the composition of
$R_+\colon \un \to U_- \otimes U_+$ and $\mu g \otimes \iota$, for
some $\lambda, \mu \in \C^\times$. Then the conjugate equations and the
previous proposition imply the assertion.  \ep

\smallskip

\subsection{Quantum case}
\label{sec:quantum-case}

Next let us move on to the quantum case with $q > 0$, $q\ne1$.  We follow the
conventions of~\cite{neshveyev-tuset-book} for the quantized
universal enveloping algebra $U_q(\mathfrak{so}_{2n})$: it is
generated by elements $E_1, \ldots, E_n$, $F_1, \ldots, F_n$, and $K_1^{\pm 1}, \ldots,
K_n^{\pm 1}$ subject to a standard set of relations based on the root
datum.  For our purpose, it is convenient to consider the elements $X_i =
K_i^{-\half} E_i$ and $Y_i = X_i^* = F_i K_i^{\half}$, so that one has
\begin{align}\label{eq:q-coprod-X-Y}
  \Dhat_q(X_i) &= X_i \otimes K_i^{-\half} + K_i^{\half} \otimes X_i,& \Dhat_q(Y_i) &= Y_i \otimes K_i^{-\half} + K_i^{\half} \otimes Y_i.
\end{align}

Keeping the Hermitian forms on $V$ and $U_\pm$, we have unitary representations of~$U_q(\so_{2n})$ on these spaces given
by the same formulas for the elements $X_i$ and $Y_i$ as in the case $q=1$, see~\cite{MR1068378}. To see this, consider
$H_i = [X_i, Y_i]$ computed in $\so_{2n}$, so that we have
$H_i = e_{n+i} \wedge e_i - e_{n+i+1} \wedge e_{i+1}$ for $i < n$ and
$H_n = e_n \wedge e_{2n} + e_{n-1} \wedge e_{2n-1}$.  We then observe
that since the weights appearing in $V$ are of the form $\pm L_i$, the
elements $H_i$ act with eigenvalues $0, \pm 1$, and therefore
$$
\frac{q^{H_i} - q^{-H_i}}{q - q^{-1}}=H_i\ \ \text{on}\ \ V.
$$
We let $K_i$ act on $V$ as $q^{H_i}$. The quantum Serre relations are trivially satisfied, e.g.,~because $X_i^2=0$ on~$V$ and therefore the quantum Serre relations for the elements $X_i$ are equivalent to the classical ones. Similar arguments apply to $U_+$ and $U_-$, since by \eqref{eweight} the elements $H_i$ again act with eigenvalues $0, \pm 1$.

We thus need to find  intertwiners $V \to U_+ \otimes U_+$
and $U_- \to U_+ \otimes V$ with respect to the new
coproduct~\eqref{eq:q-coprod-X-Y} for the same representations of
$X_i$ and $Y_i$ on $U_\pm$ and $V$ as for $q=1$. The formulas are similar to \eqref{eq:V-hwv-in-S-plus-squared} and \eqref{eq:S-minus-in-S-plus-times-V} but involve extra factors of $q$. More precisely, put $\sigma^q_i(X)=(-q)^{I_i(X)}$ for $1\le i\le 2n$. Then replacing $\sigma_i(X)$ by $\sigma^q_i(X)$ in~\eqref{eq:V-hwv-in-S-plus-squared}  we get vectors $\tilde e_i^q\in U_+\otimes U_+$ defining an embedding $V\to U_+\otimes U_+$ which maps $e_i\in V$ into $\tilde e^q_i$. This is proved similarly to the case $q=1$. For example, we compute $\Dhat_q(X_i)\tilde e^q_1$ for $2\le i\le n-1$ as
$$
\sum_{\substack{(X',X'')\in\Omega_1\\ i\notin X',\ i+1\in X'}} q^{-\half}\sigma^q_1(X') e_{X'\cup\{i\}\setminus\{i+1\}}\otimes e_{X''} + \sum_{\substack{(Y',Y'')\in\Omega_1\\ i\in Y',\ i+1\notin Y'}} q^{\half}\sigma^q_1(Y')e_{Y'}\otimes e_{Y''\cup\{i\}\setminus\{i+1\}}.
$$
As in the case $q=1$, if we take some $X'$ in the first sum and then consider $Y'=X'\cup\{i\}\setminus\{i+1\}$ in the second, then the corresponding terms cancel, since $I_1(X')=I_1(Y')+1$. Thus $\Dhat_q(X_i)\tilde e^q_1=0$.

Concerning \eqref{eq:S-minus-in-S-plus-times-V}, a correct formula for the lowest weight vector defining an embedding $U_-\to U_+\otimes V$
is
$$
\sum^n_{i=1}q^{i}e_i\otimes e_{n+i},
$$
while an embedding $U_-\to V\otimes U_+$ can be defined using
the lowest weight vector
$$
\sum^n_{j=1}q^{-j}e_{n+j}\otimes e_{j}.
$$
Again, this is easy to check similarly to the case $q=1$.

\begin{proof}[Proof of Theorem~\ref{thm:spin-2n-n-odd-still-good} for $q
  \neq 1$]
We need to
establish an analogue of Proposition~\ref{prop:g-star-h-compos}. The only difference from the case $q=1$ is that in the analogue of equation~\eqref{eq:inn-prod-lwv} we get an extra factor $q^{i-j+I_{n+i}(X_{i,j})+I_{n+j}(\{i\})}$. Since this factor is positive, we still have $g^* h \neq 0$, and the rest of the argument is the same
as for $q=1$.
\end{proof}

\bigskip

\end{document}